\documentclass[11pt]{article}
\usepackage{amstext,amssymb,amsmath,amsbsy}

\usepackage{tikz}
\usepackage{hyperref}
\usepackage{amscd}
\usepackage{amsfonts}
\usepackage{indentfirst}
\usepackage{verbatim}
\usepackage{amsmath}
\usepackage{amsthm}
\usepackage{enumerate}
\usepackage{graphicx}
\usepackage{subfig}
\usepackage{color}
\usepackage[OT1]{fontenc}
\usepackage[latin1]{inputenc}
\usepackage[english]{babel}
\usepackage{amssymb}
\newtheorem{theorem}{Theorem}
\newtheorem{lemma}{Lemma}

\newtheorem{proposition}{Proposition}

\newtheorem{remark}{Remark}

\newtheorem{definition}{Definition}
\newcommand{\tr}{^\mathsf{T}}
 
\newcommand{\proofend}{\hfill $\Box$ }

\newcommand{\dsp}{\displaystyle}

\newcommand{\newtext}[1]{\textcolor{black}{#1}}

\newcommand{\newtextb}[1]{\textcolor{black}{#1}}

\newcommand{\eps}{\varepsilon}

\newcommand{\mN}{\mathbb{N}}

\newcommand{\mR}{\mathbb{R}}

\newcommand{\hbK}{\hat {\bf K}}

\newcommand{\D}{{\cal T}}
\newcommand{\hD}{\hat {\cal T}}

\date{\empty}

\title{Optimal time for the \newtext{controllability} of linear hyperbolic systems in one dimensional space}
\author{Jean-Michel Coron \thanks{Sorbonne Universit\'{e}, Universit\'{e} Paris-Diderot SPC, CNRS, INRIA, Laboratoire Jacques-Louis Lions, \'{e}quipe Cage, Paris, France, coron@ann.jussieu.fr.} \and Hoai-Minh Nguyen \thanks{Ecole Polytechnique F\'ed\'erale de Lausanne, EPFL,  SB MATHAA CAMA, Station 8,  CH-1015 Lausanne, Switzerland,  hoai-minh.nguyen@epfl.ch. }}

\begin{document}
\maketitle

\begin{abstract}
We are concerned about the \newtext{controllability} of a general linear hyperbolic system of the form  $\partial_t w (t, x) =  \Sigma(x) \partial_x w (t, x) + \gamma C(x) w(t, x) $ ($\gamma \in \mR$) in one space dimension  using boundary controls on one side. More precisely, we establish  the optimal time for the \newtext{null and exact} controllability of the hyperbolic system for generic $\gamma$. We also present examples which yield that the generic requirement is necessary.
In the case of constant $\Sigma$ and of two  positive directions, we
prove that the null-controllability is attained for any time greater than the optimal time for all $\gamma \in \mR$ and for  all  $C$ which is analytic if  the slowest negative direction can be alerted by  {\it both} positive directions. We also  show that the null-controllability is attained at the optimal time by a feedback law when  $C \equiv 0$.   Our approach is based on the  backstepping method paying  a special attention on the  construction of the kernel and the selection of controls.

\end{abstract}


\noindent{\bf AMS}:  35F05, 35F15, 35B37, 58G16, 93C20.


\section{Introduction}

Linear hyperbolic systems in one dimensional space are frequently used
in modelling of many systems such as traffic flow, heat exchangers,  and fluids in open channels. The
stability and boundary stabilization of these hyperbolic systems
have been studied intensively in the literature, see,  e.g.,   \cite{BC} and the references therein. In this paper, we are concerned about the optimal time for the \newtext{null-controllability and exact controllability} of such systems using boundary controls on one side.  More precisely, we consider the system
\begin{equation}\label{Sys-1}
\partial_t w (t, x) =  \Sigma(x) \partial_x w (t, x) + \gamma C(x) w(t, x) \mbox{ for } (t, x)  \in \mR_+ \times (0, 1).
\end{equation}
Here $w = (w_1, \cdots, w_n)\tr : \mR_+ \times (0, 1) \to \mR^n$ ($n \ge 2$), $\gamma \in \mR$,   $\Sigma$ and $C$ are   $(n \times n)$ real matrix-valued functions defined in $[0,1]$. We assume that for every $x \in [0, 1]$,  $\Sigma(x)$ is diagonal with $m \ge 1$  distinct positive eigenvalues and $k = n - m \ge 1$  distinct negative eigenvalues. Using Riemann coordinates, one might assume that $\Sigma(x)$ is of the form
\begin{equation}\label{form-A}
\Sigma(x) = \mbox{diag} \big(- \lambda_1(x), \cdots, - \lambda_{k}(x),  \lambda_{k+1}(x), \cdots,  \lambda_{n}(x) \big),
\end{equation}
where
\begin{equation}\label{relation-lambda}
-\lambda_{1}(x) < \cdots <  - \lambda_{k} (x)< 0 < \lambda_{k+1}(x) < \cdots < \lambda_{k+m}(x).
\end{equation}
Throughout the paper, we assume that
\begin{equation}\label{cond-lambda}
\mbox{$\lambda_i$  is Lipschitz on $[0, 1]$  for $1 \le i \le n \,  (= k + m)$.}
\end{equation}
We are interested in the following type of boundary conditions and boundary controls. The boundary conditions at $x = 0$ are given by
\begin{equation}\label{bdry:x=0-w}
(w_1, \cdots, w_k)\tr (t, 0)  = B(w_{k+1}, \cdots, w_{k+m})\tr (t, 0) \mbox{ for } t \ge 0,
\end{equation}
for some $(k \times m)$ real {\it constant} matrix $B$, and the boundary controls  at $x = 1$ are
\begin{equation}\label{control:x=1-w}
w_{k+1}(t, 1) = W_{k+1}(t),  \quad \dots, \quad w_{k + m}(t, 1) = W_{k+m}(t) \mbox{ for } t \ge 0,
\end{equation}
where $W_{k +1}, \dots, W_{k + m}$ are controls. Our goal is to obtain  the optimal time for  the null-controllability and exact controllability of \eqref{Sys-1}, \eqref{bdry:x=0-w}, and \eqref{control:x=1-w}. \newtext{Let us recall that the control system \eqref{Sys-1}, \eqref{bdry:x=0-w}, and \eqref{control:x=1-w} is null-controllable (resp. exactly controllable) at the time $T>0$ if, for every initial data $w_0: (0,1)\to \mathbb{R}^n$ in $[L^2(0,1)]^n$ (resp. for every initial data $w_0: (0,1 )\to \mathbb{R}^n$ in $[L^2(0,1)]^n$  and for every (final) state $w_T:  (0,1 )\to \mathbb{R}^n$  in $[L^2(0,1)]^n$), there is a control $W=(W_{k+1},\ldots,W_{k+m})\tr :(0,T)\to \mathbb{R}^m$ in $[L^2(0,T)]^m$ such that the solution
of \eqref{Sys-1}, \eqref{bdry:x=0-w}, and \eqref{control:x=1-w} satisfying $w(0,x)=w_0(x)$ vanishes (resp. reaches $w_T$) at the time $T$: $w(T,x)=0$ (resp. $w(T, x)= w_T(x)$).}
\medskip
Set
\begin{equation}\label{def-tau}
\tau_i :=  \int_{0}^1 \frac{1}{\lambda_i(\xi)}  \, d \xi  \mbox{ for } 1 \le i \le n
\end{equation}
and
\begin{equation}\label{def-Top}
T_{opt} := \left\{ \begin{array}{cl}  \dsp \max \big\{ \tau_1 + \tau_{m+1}, \dots, \tau_k + \tau_{m+k}, \tau_{k+1} \big\} & \mbox{ if } m \ge k, \\[6pt]
\dsp \max \big\{ \tau_{k+1-m} + \tau_{k+1},  \tau_{k+2-m} + \tau_{k+2},  \dots, \tau_{k} + \tau_{k+m} \big\} &  \mbox{ if } m < k.
\end{array} \right.
\end{equation}

The first  result in this paper, which implies in particular  that one can reach the null-controllability of  \eqref{Sys-1}, \eqref{bdry:x=0-w}, and \eqref{control:x=1-w} at the time $T_{opt}$ for generic $\gamma$ (and $B$), is

\begin{theorem}\label{thm1} Assume that \eqref{relation-lambda} and \eqref{cond-lambda} hold. We define
\begin{equation}
{\cal B}: = \big\{B \in \mR^{k \times m}; \mbox{ such that  \eqref{cond-B-1} holds for \newtextb{ $1 \le i \le  \min\{k, m-1\}$}} \big\},
\end{equation}
where
\begin{multline}\label{cond-B-1}
\mbox{ the $i \times i$  matrix formed from the last $i$ columns and the last $i$ rows of $B$  is invertible.}
\end{multline}
Then
\begin{enumerate}
\item In the case $m = 1$, there exists a (linear) time independent feedback which yields the null-controllability at the time $T_{opt}$.

\item In the case $m  = 2$,  if  $B \in {\cal B}$, $B_{k1} \neq 0$, $\Sigma$ is constant,  and  $(T_{opt} =  \tau_k + \tau_{k+2} = \tau_{k-1} + \tau_{k+1}$ if $k \ge 2$ and $T_{opt} =  \tau_1 + \tau_3 = \tau_2$ if $k = 1)$, then there exists a non-zero constant matrix $C$ such that the system is {\bf not} null-controllable  at  the time  $T_{opt}$.

\item In the case $m \ge 2$, we have i) for each $B \in {\cal B}$, outside a discrete set of $\gamma$ in $\mR$, the control system \eqref{Sys-1}, \eqref{bdry:x=0-w}, and \eqref{control:x=1-w} is null-controllable at the time $T_{opt}$, and ii) for each $\gamma$ outside a discrete set in $\mR$, outside a set of zero measure of $B$ in ${\cal B}$, the control system \eqref{Sys-1}, \eqref{bdry:x=0-w}, and \eqref{control:x=1-w} is null-controllable at the time $T_{opt}$.

\end{enumerate}
\end{theorem}

Theorem \ref{thm1} is proved in Section~\ref{sect-thm1}. \newtextb{The optimality of $T_{opt}$ is shown in Proposition~\ref{pro-C} for $C \equiv 0$ (see also Remark \ref{rem-OPT}).}

\begin{remark}\rm In Proposition~\ref{pro-T2} in Section~\ref{sect-FP},
we present a null-controllability result, which holds for all $\gamma$ and $B \in {\cal B}$,  for a time which is larger than $T_{opt}$ but smaller than  $T_2$ defined in \eqref{def-T2} for $m \ge 2$.
\end{remark}

\newtext{Concerning the exact controllability, we have the following theorem whose proof is just a straightforward modification of the one of Theorem~\ref{thm1} (see Remark~\ref{rem-proof-thm-exact}).
\begin{theorem}\label{thm-exact}
\newtext{Assume that $m \ge k \ge 1$,  \eqref{relation-lambda} and \eqref{cond-lambda} hold. Define}
\begin{equation}
{\cal B}_e: = \big\{B \in \mR^{k \times m}; \mbox{ such that  \eqref{cond-B-1} holds for \newtext{ $1 \le i \le  k$}} \big\},
\end{equation}
Then,
i) for each $B \in {\cal B}_e$, outside a discrete set of $\gamma$ in $\mR$, the control system \eqref{Sys-1}, \eqref{bdry:x=0-w}, and \eqref{control:x=1-w} is exactly controllable at the time $T_{opt}$, and ii) for each $\gamma$ outside a discrete set in $\mR$, outside a set of zero measure of $B$ in ${\cal B}_e$, the control system \eqref{Sys-1}, \eqref{bdry:x=0-w}, and \eqref{control:x=1-w} is exactly controllable at the time $T_{opt}$.
\end{theorem}}


\newtext{\begin{remark} \rm In the case $k = m = 1$, the result of Theorem~\ref{thm-exact} holds for all $\gamma$ and $B \in B_e$, which was already proved in \cite{Russell78}. Our proof can be modified to obtain this result.
\end{remark}}

In the case where  $k \ge 1$, $m =  2$, $\Sigma$ is constant,  $B \in {\cal B}$, and $B_{k1} \neq 0$, we show that the system is {\it null-controllable} for any time greater than $T_{opt}$ for all $\gamma \in \mR$ and  $C$ analytic. More precisely, we have

\begin{theorem}\label{thm2} Let $k \ge 1$,  $m=2$, and  $T> T_{opt}$. Assume that \eqref{relation-lambda} holds, $B \in {\cal B}$   and $B_{k1} \neq 0$,  $\Sigma$ is constant,  and $C$ is analytic  on  $[0, L]$ \footnote{This means that $C$ is analytic in a  neighborhood of $[0, L]$.} where
\begin{equation}\label{def-L}
L =  \frac{\rho_k}{ \rho_k - 1} \quad \mbox{ with } \quad  \rho_k = \left\{ \begin{array}{cl} \frac{\lambda_{k+2}}{ \lambda_{k+1}}  & \mbox{ if } k = 1, \\[6pt]
\min\big\{ \mathop{\min}_{ 1 \le j < i \le k } \frac{\lambda_j}{\lambda_i}, \frac{\lambda_{k+2}}{\lambda_{k+1}} \big\} & \mbox{ if } k \ge 2.
\end{array}\right.
\end{equation}
Then the system is null-controllable at the time $T$. \newtext{Similarly, if in addition that  $m \ge k$ and $B \in {\cal B}_e$, then the system is exactly controllable at the time $T$.}
\end{theorem}

Theorem~\ref{thm2} is proved in Section~\ref{sect-thm2}.

In the case $C \equiv 0$, we can prove  that $T_{opt}$ is the optimal time for the null-controllability of the considered system via a linear time independent feedback law. More precisely, we have

\begin{proposition} \label{pro-C} Assume that $C \equiv 0$ and \eqref{cond-B-1} holds for $1 \le i \le \min\{k, m - 1\} $.
There exists a {\bf linear time independent feedback} which yields the null-controllability at the time $T_{opt}$. Assume in addition that \eqref{cond-B-1} holds for $i =  \min\{k, m\}$,  then, for any $T < T_{opt}$, there exists an initial datum such that $u(T, \cdot) \not \equiv 0$ for every control.
\end{proposition}

Proposition~\ref{pro-C} is proved in   Section~\ref{sect-pro-C}.

We now briefly describe the method used in the proofs. Our approach relies on  backstepping due to Miroslav Krstic and his coauthors (see also Remark~\ref{rem-Ref}). \newtextb{More precisely, we make the following change of variables
$$
u(t, x) = w(t, x) - \int_0^x K(x, y) w(t, y) \, d y
$$
for some kernel $K: \D = \big\{(x, y) \in (0, 1)^2; 0  <  y <  x \big\}
 \to \mR^{n}$. The idea is to choose $K$ in such a way that the controllability of the target system of $u$ is easier to investigate. In our case, $K$ is chosen so that
\eqref{eq-u} holds with $K(x, 0)$ having appropriate properties, see in particular \eqref{def-Q} and  \eqref{def-S}.}

The use of backstepping method to obtain the null-controllability for hyperbolic systems in one dimension  was initiated  in \cite{Coron13} for the case $m = k = 1$.
This approach has been  developed later on for more general hyperbolic system in \cite{2015-Hu-Di-Meglio-Vazquez-Krstic-preprint, Auriol16, Coron17}. In \cite{Coron13}, the optimal time $T_{opt}$ is obtained for the case $m  = k = 1$. In \cite{2015-Hu-Di-Meglio-Vazquez-Krstic-preprint}, the authors considered the case where $\Sigma$ is constant. They  obtained the null-controllability for the time
\begin{equation}\label{def-T1}
T_1 : = \tau_k + \sum_{l = 1}^m \tau_{k + l}.
\end{equation}
 It was  later showed in \cite{Auriol16, Coron17} that one can reach the null-controllability at the time
\begin{equation}\label{def-T2}
T_2 := \tau_k + \tau_{k+1}.
\end{equation}
In \cite{2015-Hu-Di-Meglio-Vazquez-Krstic-preprint, Auriol16, Coron17}, one does not require any conditions on $B$ and the optimal time in this case is $T_2$.
With the convention \eqref{relation-lambda}, it is clear that
$$
T_{opt} \le  T_2 \le  T_1, 
$$
and
$$
T_2 < T_1 \mbox{ if } m > 1 \quad \mbox{ and } \quad T_{opt} < T_2 \mbox{ if } m > 1 \mbox{ or } k > 1.
$$
When $C \equiv 0$, Long Hu \cite{LongHu} established the  \newtext{exact} controllability for quasilinear systems, i.e.,  $A = A(u)$,  in the case $m \ge k$ for the time
$$
T_3 := \max\{\tau_{k+1}, \tau_{k} + \tau_{m+1}\}
$$
under a condition on $B$, \newtext{which is equivalent to \eqref{cond-B-1} with $i = k$ in our setting}.  It is clear in the case $m \ge k$ that
$$
T_{opt} \le T_3  \le T_2, \quad T_3 = T_{opt} \mbox{ if } k =1,  \quad \mbox{ and } \quad T_{opt} < T_3 \mbox{ if } k >  1.
$$

\newtext{In the linear case, the null controllability was established for the time $T_{2}$ without any assumption on $B$ and the exact controllability was obtained in the case $m = k$  under a condition which is different but has some similar features to condition \eqref{cond-B-1} with $i = k$ in \cite[Theorem 3.2]{Russell78}. In the quasilinear case with $m \ge k$,  the exact controllability was derived in   \cite[Theorem 3.2]{Li00} (see also \cite{LiRao02}) for $m \ge k$ and  for the time $T_2$ under a condition which is equivalent to \eqref{cond-B-1} with $i = k$ in our setting.}

Theorems~\ref{thm1} \newtext{and \ref{thm-exact}},  and Proposition~\ref{pro-C} confirm that generically the optimal time to reach the null-controllability for the system in \eqref{Sys-1}, \eqref{bdry:x=0-w}, and \eqref{control:x=1-w} is $T_{opt}$. Condition $B \in {\cal B}$ (resp. $B \in {\cal B}_e$) is very natural to obtain the null-controllability (\newtext{resp. exact-controllability}) at $T_{opt}$ (see Section~\ref{sect-thm1-part2} for details) which roughly speaking  allows to  use the $l$ controls  $W_{k+ m - l + 1}, \cdots, W_{k+m}$  to control $u_{k-l +1}, \cdots, u_{k}$ for $1 \le l \le \min\{k, m\}$ (the possibility to implement $l$ controls corresponding to the fastest  positive speeds to control  $l$ components corresponding to the lowest negative speeds).

In comparison with the previous works mentioned above, our analysis contains two new ingredients. First, after transforming  the system into a new one (target system) via backstepping method as usual,  we carefully choose the control varying with respect to time
 so that the zero state  is reachable at  $T_{opt}$;  in the previous works, the zero controls were used for the target system.  Secondly, the boundary conditions  of the kernel obtained from the backstepping approach given in this paper are different from  the known ones.  Our idea is to explore as much as possible the boundary conditions of the kernel to make the target system as simple as possible from the control point of view.

\begin{remark}\label{rem-Ref} \fontfamily{m} \selectfont
The backstepping method has been also used to stabilize the  wave equation  \cite{2008-Krstic-Guo-Balogh-Smyshlyaev, 09-krstic-smyshlyaev, 2010-Smyshlyaev-Cerpa-Krstic-SICON}, the parabolic equations in \cite{Smyshlyaev04,05-krstic-smyshlyaev},  nonlinear parabolic equations  \cite{Vazquez08}. The standard backstepping approach relies on the Volterra transform of the second kind. In some situations, more general transformations are  considered as for Korteweg-de Vries equations  \cite{13-Cerpa-Coron-TAC}, Kuramoto--Sivashinsky equations \cite{2015-Coron-Lu-JDE}, and Schr\"odinger's equation \cite{coron-gagnon16}. The use of backstepping method  to obtain the null-controllability of the heat equation is given in \cite{Coron-Nguyen}.  A concise introduction of this method applied to numerous partial differential equations can be found in \cite{Krstic08}.
\end{remark}

The paper is organized as follows. In Section~\ref{sect-backstepping}, we apply the backstepping approach to derive the target system and the equations for the kernel.  Section~\ref{sect-preliminaries} is devoted to some properties on the control systems and the kernel. The proofs of Theorems \ref{thm1} and \ref{thm2} are presented in
Sections~\ref{sect-thm1} and \ref{sect-thm2} respectively.   A null-controllability result which holds for all $\gamma$ and $B \in {\cal B}$ is given in Section~\ref{sect-FP}.
In Section~\ref{sect-pro-C}, we  present the proof of Proposition~\ref{pro-C}.

\medskip
\noindent{\bf Acknowledgement:} The authors are grateful to the Institute for Theoretical Studies, ETH Z\"{u}rich for the hospitality and the support. They are also partially supported by  ANR Finite4SoS ANR-15-CE23-0007. They thank Amaury Hayat and Long Hu for useful comments.

\section{A change of variables via backstepping approach. Systems of the kernel and the target}\label{sect-backstepping}


In what follows, we assume that $\gamma =1$, the general case can be obtained from this case by replacing $C$ by $\gamma C$. \newtextb{As in \cite[Section 3]{BC0}, \cite[Section 4]{Coron13}, and \cite[Section 3]{hu15}},  without loss of generality, one can assume  that  $C_{ii}(x)  = 0$ for $1 \le i \le n$. The key idea of backstepping approach is to make the following change of variables
\begin{equation}\label{backstepping}
u(t, x) = w(t, x) - \int_0^x K(x,y) w(t, y) \, dy,
\end{equation}
for some kernel $K: \D \to \mR^{n \times  n}$  which is chosen in such a way that the system for $u$ is easier to control. Here
\begin{equation}\label{def-D}
\D = \big\{(x, y) \in (0, 1)^2; 0  <  y <  x \big\}.
\end{equation}

To determine/derive the equations for  $K$, we first compute $\partial_t u (t, x) - \Sigma(x) \partial_x u(t, x)$. Taking into account \eqref{backstepping},
we formally have \footnote{We assume here that $u$, $w$, and $K$ are smooth enough so that the below computations make sense.}
\begin{align*}
\partial_t u (t, x) =&  \partial_t w(t, x) - \int_0^x K(x,y) \partial_t w(t, y) \, dy \\[6pt]
= &  \partial_t w(t, x) - \int_0^x \Big[K(x,y) \big( \Sigma(y) \partial_y w(t, y) + C(y) w(t, y) \big) \Big]\, dy  \quad (\mbox{by } \eqref{Sys-1})\\[6pt]
= &  \partial_t w(t, x) - K(x, x) \Sigma(x) w(t, x) + K(x, 0) \Sigma(0) w(t, 0) \\[6pt]
& + \int_0^x \Big[ \partial_y  \big(K(x,y)  \Sigma(y) \big) w(t, y) - K(x, y) C(y) w(t, y) \Big] \, dy \quad (\mbox{by integrating by parts})
\end{align*}
and
\begin{align*}
\partial_x u(t, x) = \partial_x w(t, x) - \int_0^x \partial_x K(x,y) w(t, y) \, dy - K(x,x) w(t, x).
\end{align*}
It follows from \eqref{Sys-1} that
\begin{multline}
\partial_t u (t, x) - \Sigma(x) \partial_x u(t, x) =  \Big( C(x) - K(x,x) \Sigma(x)  + \Sigma(x) K(x, x) \Big) w(t, x) + K(x, 0) \Sigma(0) u(t, 0) \\[6pt]
+ \int_0^x \Big[ \partial_y K(x, y) \Sigma (y) + K(x, y)  \Sigma'(y)  - K(x, y) C(y) + \Sigma(x) \partial_x K(x, y) \Big] w(t, y) \, dy.
\end{multline}

We search a  kernel $K$ which satisfies  the following two conditions
\begin{equation}\label{equation-K}
\partial_y K(x, y) \Sigma (y) +  \Sigma(x) \partial_x K(x, y) +  K(x, y)  \Sigma'(y) - K(x, y) C(y)= 0 \mbox{ in } \D
\end{equation}
and
\begin{equation}\label{bdry1-K}
{\cal C}(x): = C(x) - K(x,x) \Sigma(x)  + \Sigma(x) K(x, x) = 0 \mbox{ for } x \in (0, 1),
\end{equation}
so that one formally has
\begin{equation}\label{eq-u}
\partial_t u (t, x) =  \Sigma(x) \partial_x u (t, x) + K(x, 0) \Sigma(0)  u(t, 0)  \mbox{ for } (t, x)  \in \mR_+ \times (0, 1).
\end{equation}
In fact, such a $K$ exists so that \eqref{eq-u} holds (see Proposition~\ref{pro-eq-u}).  We have
\begin{multline}\label{eq-Kij}\mbox{the $(i, j)$ component of the matrix $\partial_y K (x, y) \Sigma(y) +\Sigma(x) \partial_x K(x, y)$ is} \\[6pt]
a_{ij} (y) \partial_{y} K_{ij}(x, y) + b_{ij} (x)\partial_x K_{ij}(x, y),
\end{multline}
where
\begin{equation}\label{def-ab}
\big(a_{ij} (y), b_{ij} (x) \big)  = \left\{ \begin{array}{cl}
\big(-\lambda_j (y), - \lambda_i (x) \big)   & \mbox{ if } 1 \le i, j \le k, \\[6pt]
\big(\lambda_{j} (y),  -\lambda_i(x)\big)  & \mbox{ if } 1 \le i \le k < k+1 \le  j \le k + m, \\[6pt]
\big(\lambda_{j} (y),  \lambda_i (x)\big)  & \mbox{ if } k+1 \le i, j \le k + m, \\[6pt]
\big(-\lambda_j (y),  \lambda_i (x)\big)  & \mbox{ if } 1 \le j \le k < k+1 \le  i \le k + m.
\end{array}\right.
\end{equation}

We  denote
\begin{equation*}
\Gamma_1 = \big\{ (x, x); x \in (0, 1)  \big\}, \quad \Gamma_2 =  \big\{(x, 0); x \in (0, 1)  \big\}, \quad \mbox{ and } \quad  \Gamma_3 =  \big\{(1, y); y \in (0, 1)  \big\}.
\end{equation*}

\begin{remark}\rm
By the characteristic method, it is possible to impose the following boundary conditions for $K_{ij}$ when $\Sigma $ is constant:
\begin{itemize}
\item On  $\Gamma_1$ if $a_{ij}/b_{ij}\le 0$, see case $a)$ in Figure \ref{figure-2}.

\item On both $\Gamma_1$ and $\Gamma_2$ if $0 < a_{ij}/b_{ij} < 1$,  see case $b)$ in Figure \ref{figure-2}.

\item On $\Gamma_1$ and $\Gamma_3$ if  $ a_{ij}/b_{ij}> 1$, see case $c)$ in Figure \ref{figure-2}.

\item On $\Gamma_2$  if  $a_{ij}/b_{ij} = 1$,  see case $d)$ in Figure \ref{figure-2}.

\end{itemize}
\end{remark}

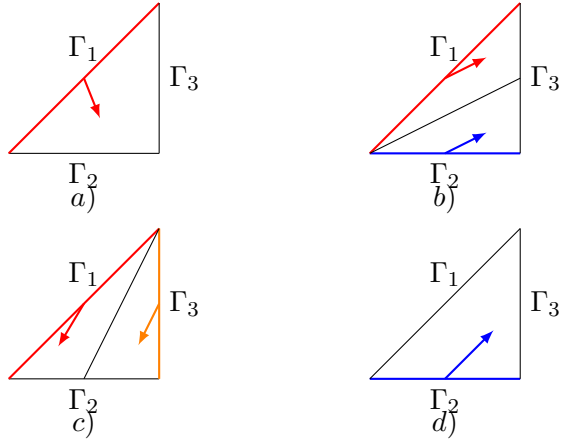
\begin{figure}
\centering
\begin{tikzpicture}[scale=2]


\draw[] (0,0) -- (1,0);
\draw[red, thick] (0,0) -- (1,1);
\draw[] (1,0) -- (1,1);

\draw (0.5, 0.55) node[above]{$\Gamma_1$};
\draw (0.5, 0.0) node[below]{$\Gamma_2$};
\draw (1.0, 0.5) node[right]{$\Gamma_3$};

\draw [-latex, red, thick, shorten >= 0.50cm] ((0.5, 0.5) --(0.7,0);

\draw (0.5, -0.3) node{$a)$};


\newcommand\x{2.4}

\newcommand\y{-1.5}

\draw[blue, thick] (0+\x,0) -- (1+\x,0);
\draw[red, thick] (0+\x,0) -- (1+\x,1);
\draw[] (1+\x,0) -- (1 + \x,1);
\draw[] (0 + \x,0) -- (1 + \x, 0.5);

\draw (0.5 +\x, 0.55) node[above]{$\Gamma_1$};
\draw (0.5 +\x, 0.0) node[below]{$\Gamma_2$};
\draw (1.0 +\x, 0.5) node[right]{$\Gamma_3$};

\draw [-latex, blue, thick, shorten >= 0.5cm] ((0.5 +\x, 0.0) --(1+\x,0.25);
\draw [-latex, red, thick, shorten >= 0.5cm] ((0.5 +\x, 0.5) --(1+\x, 0.75);

\draw (0.5 + \x, -0.3) node{$b)$};


\draw[] (0,0 + \y) -- (1,0+ \y);
\draw[red, thick] (0,0+ \y) -- (1,1+ \y);
\draw[orange, thick] (1,0+ \y) -- (1,1+ \y);
\draw[] (1,1+ \y) -- (0.5,0+ \y);

\draw (0.5, 0.55 + \y) node[above]{$\Gamma_1$};
\draw (0.5, 0.0 + \y) node[below]{$\Gamma_2$};
\draw (1.0, 0.5 + \y) node[right]{$\Gamma_3$};

\draw [-latex, orange, thick, shorten >= 0.5cm] ((1, 0.5+ \y) --(0.75,0+ \y);
\draw [-latex, red, thick, shorten >= 0.5cm] ((0.5, 0.5 + \y) --(0.2, 0+ \y);

\draw (0.5, -0.3 + \y) node{$c)$};


\draw[blue, thick] (0 + \x,0 + \y) -- (1 + \x,0 + \y);
\draw[] (0 + \x,0 + \y) -- (1+ \x,1+ \y);
\draw[] (1 + \x,0 + \y) -- (1+ \x,1+ \y);

\draw (0.5+ \x, 0.55+ \y) node[above]{$\Gamma_1$};
\draw (0.5+ \x, 0.0+ \y) node[below]{$\Gamma_2$};
\draw (1.0 + \x, 0.5+ \y) node[right]{$\Gamma_3$};

\draw [-latex, blue, thick, shorten >= 0.5cm] ((0.5+ \x, 0+ \y) --(1+ \x,0.5+ \y);
\draw (0.5 + \x, -0.3 + \y) node{$d)$};

\end{tikzpicture}
\caption{The characteristic vectors of $K_{ij}$ in the case $\Sigma$ is constant: a) in the case  $a_{ij}/b_{ij} < 0$ ($1 \le i \le k< k+1 \le j \le k+ m$ or $1 \le j \le k < k+1 \le i \le k+ m$), b) in the case $0 < a_{ij}/b_{ij} < 1$ ($1 \le i < j \le k$ or $k+1 \le j < i \le k+ m$), c) in the case $a_{ij}/b_{ij} > 1$ ($1 \le j < i \le k$ or $k + 1 \le i < j \le k+ m$), and d) in the case $a_{ij}/b_{ij} = 1$ ($1 \le i = j \le k+m$).}\label{figure-2}
\end{figure}

To impose (appropriate) boundary conditions of $K$ on $\Gamma_2$ so that the system for $u$ is  simple, we investigate the term $K(x, 0) \Sigma(0) u(t, 0)$. Set
\begin{equation}\label{def-Q}
Q := \left(\begin{array}{cccc} 0_k & B \\[6pt]
0_{m, k} & I_m
\end{array}\right).
\end{equation}
Here and in what follows, $0_{i, j}$  denotes the zero matrix  of size $i \times j$, and $0_i$ and $I_i$ denotes the zero matrix and the identity matrix of the size $i \times i$ for $i, \, j \in \mN$. Using the boundary conditions at $x= 0$ in \eqref{bdry:x=0-w} and the fact that $u(t, 0) = w(t, 0)$, we obtain
$$
u(t, 0) = Q u(t, 0).
$$
It follows that
$$
K(x, 0) \Sigma(0) u(t, 0) = K(x, 0) \Sigma(0) Q u(t, 0).
$$
We have, by the definition of $Q$ in  \eqref{def-Q},
$$
\Sigma (0) Q = \left(\begin{array}{cccc} 0_k & \Sigma_-(0) B \\[6pt]
0_{m, k} & \Sigma_+(0)
\end{array}\right).
$$
Here and in what follows,  we define, for $x \in [0, 1]$,
$$
\mbox{$\Sigma_- (x) := \mbox{diag } \big(-\lambda_1(x), \dots, -\lambda_k(x) \big)$ and  $\Sigma_+ (x) : = \mbox{diag }\big(\lambda_{k+1}(x), \dots, \lambda_{k+m}(x) \big)$.}
$$
Denote
$$
K(x, 0) =  \left(\begin{array}{cccc} K_{--}(x) & K_{-+}(x)  \\[6pt]
K_{+-}(x) & K_{++}(x)
\end{array}\right),
$$
where $K_{--}$, $K_{-+}$, $K_{+-}$, and $K_{++}$ are matrices of size $k \times k$, $k \times m$, $m \times k$, and $m \times m$, respectively.
Set
\begin{equation}\label{def-S-*}
S(x): = K(x, 0) \Sigma (0) Q.
\end{equation}
We have
\begin{equation}\label{def-S}
S(x) =  \left(\begin{array}{cccc} 0_k & K_{--}(x) \Sigma_-(0) B + K_{-+} (x)\Sigma_+ (0)\\[6pt]
0_{m, k} & K_{+-} (x)\Sigma_-(0) B + K_{++} (x) \Sigma_+(0)\end{array}\right)  =  \left(\begin{array}{cccc} 0_k & S_{-+} (x)  \\[6pt]
0_{m, k} & S_{++}(x)
\end{array}\right).
\end{equation}

We impose boundary conditions for $K_{ij}$ on $\Gamma_1$, $\Gamma_2$, and $\Gamma_3$ as follows:

\begin{itemize}

\item[$BC_1)$] For $(i, j)$ with $1 \le i \neq j \le k + m$, we impose the boundary condition for $K_{ij}$ on $\Gamma_1$ in such a way that ${\cal  C}_{ij}(x) = 0$ (recall that ${\cal C}$ is define in \eqref{bdry1-K}). More precisely, we have, noting that $a_{ij} \neq b_{ij}$,
\begin{equation}\label{bdry-K-1}
K_{ij} (x,x)= C_{ij}(x) / \big(a_{ij}(x) - b_{ij}(x) \big) \mbox{ for } x \in (0, 1).
\end{equation}

\item[$BC_2)$] Set
\begin{equation*}
{\cal J} = \big\{(i, j); \; 1 \le i \le j \le k \mbox{ or } k+1 \le j \le i \le k + m \big\},
\end{equation*}
Note that if \big($i \neq j$ and $(i, j) \in {\cal J}$\big)  then $0< a_{ij}(0)/ b_{i, j}(0) < 1$ and the characteristic trajectory passing $(0,0)$ is inside $\D$ as in  case b) in Figure~\ref{figure-2}.   Using \eqref{eq-Kij} and \eqref{def-ab},   we can impose the boundary condition of $K_{ij}$ on $\Gamma_2$ with $(i, j) \in {\cal  J}$ in such a way that, for $x \in (0, 1)$,
\begin{equation}\label{cond-S-0}
K_{ij}(x, 0) = 0 \mbox{ for } 1 \le i \le j \le k
\end{equation}
and
\begin{equation}\label{cond-S}
\quad (S_{++})_{pq}(x) = 0 \mbox{ for } 1 \le q \le p \le m.
\end{equation}
These imposed conditions can be written under the form, for $(i, j) \in {\cal J}$,
\begin{equation}\label{bdry-K-2}
K_{i j}(x, 0) = \sum_{(r, s) \not \in {\cal J}} c_{i j r s}(B)K_{r s}(x, 0) \mbox{ for } x \in (0, 1),
\end{equation}
for some $c_{i j r s}(B)$ which is linear with respect to  $B$. \newtextb{Indeed, \eqref{cond-S-0} can be written under the form of \eqref{bdry-K-2} with $c_{i j r s} =0$ and for $1 \le q \le p \le m$, $K_{p,q}$ can be written under the form of \eqref{bdry-K-2} since the $(p,q)$ component of  $S_{++} = K_{+-} (x)\Sigma_-(0) B + K_{++} (x) \Sigma_+(0)$ is 0.}

\item[$BC_3)$] For $(i, j)$ with either $1 \le j < i \le k$ or $k+1 \le i < j \le k+ m$, we impose the zero boundary condition of $K_{ij}$ on $\Gamma_3$, i.e.,
\begin{equation}\label{bdry-K-3}
K_{ij}(1, y) = 0 \mbox{ for } y \in (0, 1).
\end{equation}
(Note that in this case $a_{ij}(1)/ b_{ij}(1) > 1$ and hence the characteristic trajectory passing $(1,1)$ is in $\D$ as in case  $c)$ in Figure~\ref{figure-2}).
\end{itemize}

Below are  the form  of $S (= S^{k, m})$ when $BC_2)$ is taken into account for some pairs $(k, m)$:
\begin{equation}\label{form-S-1}
 S^{2, 3}(x) = \left(\begin{array}{cccccc}
 0 & 0 & * & * & * \\[6pt]
 0 & 0 & * & * & * \\[6pt]
 0 & 0 & 0 & * & * \\[6pt]
 0 & 0 & 0 & 0 & * \\[6pt]
 0 & 0 & 0 & 0 & 0
\end{array}\right), \quad \mbox{ and } \quad
 S^{3, 2}(x) = \left(\begin{array}{cccccc}
 0 & 0 & 0 & * & * \\[6pt]
 0 & 0 & 0 & * & * \\[6pt]
 0 & 0 & 0 & * & * \\[6pt]
 0 & 0 & 0 & 0 & * \\[6pt]
 0 & 0 & 0 & 0 & 0
\end{array}\right).
\end{equation}
\newtextb{Here and in what follows, in a matrix, $*$ means that this part of that matrix can be whatever.}


\begin{remark} \label{rem-choice} \rm We here impose  \eqref{cond-S-0}  on $\Gamma_1$  and \eqref{bdry-K-3} on $\Gamma_3$. These choices are just for the simplicity of presentation. We later modify these in the proof of Theorem~\ref{thm2}.
\end{remark}

\section{Properties of the control systems and the kernel} \label{sect-preliminaries}

In this section, we establish the well-posedness of $u$, $w$, and $K$ and the unique determination of $w$ from $u$.  For notational ease, we assume that $\gamma =1$ (except in Lemma~\ref{lem2} and its proof), the general case follows easily.
We first investigate the well-posedness of $w$ and $u$ under the boundary conditions and the controls considered. We consider a more general control system, for $T> 0$,
\begin{equation}\label{sys-v}
\left\{\begin{array}{cl}
\partial_t v(t,x) = \dsp \Sigma(x) \partial_x v(t, x) + C(x) v(t, x) + D(x) v(t, 0) + f(t, x) & \mbox{ for } (t, x) \in (0, T) \times (0, 1),  \\[6pt]
v_- (t, 0) = \dsp B v_+(t, 0) + g(t) &  \mbox{ for }  t \in (0, T), \\[6pt]
v_+(t, 1) = \dsp \sum_{r=1}^R A_r (t) v\big(t, x_r \big) +   \int_{0}^1 M(t, y) v(t, y) \, dy + h(t) &  \mbox{ for }  t \in (0, T), \\[6pt]
v(t = 0, x) = v_0(x) &  \mbox{ for } x \in (0, 1),
\end{array}\right.
\end{equation}
where  $v_- = (v_1, \cdots, v_k)\tr $ and $v_+ =  (v_{k+1}, \cdots, v_{k+m})\tr  $.  Here $R\in \mN$,  $C, D: [0, 1] \to \mR^{n \times n}$, $A_r: [0, T] \to \mR^{m \times n}$, $x_r \in [0, 1]$ $(1 \le r \le R)$, $M: [0, T] \times [0, 1] \to \mR^{n \times n}$, $f \in \big[L^\infty \big((0, T) \times (0, 1) \big) \big]^n$, $g \in [L^\infty(0, T)]^k$, and $h \in [L^\infty(0, T)]^m$.
We make the following assumptions for this system
\begin{equation}\label{assumption-1}
x_r < c < 1 \mbox{ for some constant $c$},
\end{equation}
\begin{equation}\label{assumption-2}
C, D \in [L^\infty (0, 1)]^{n\times n}, \quad  A_r \in [L^\infty  (0, T)]^{n\times n},  \quad \mbox{ and } \quad
M \in [L^\infty \big((0, T) \times (0, 1) \big)]^{n\times n}.
\end{equation}
We are interested in bounded broad solutions of \eqref{sys-v} whose definition is as follows.
Extend $\lambda_i$ in $\mR$ by $\lambda_i(0)$ for $x< 0$ and $\lambda_i(1)$ for $x \ge 1$.  For $(s, \xi) \in [0, T] \times [0, 1]$, define $x_i(t, s, \xi)$ for $t \in \mR$ by
\begin{equation}\label{def-xi-1}
\frac{d}{d t} x_i(t, s, \xi) = \lambda_i \big(x_i(t, s, \xi) \big) \mbox{ and }  x_i(s, s, \xi) = \xi \mbox{ if } 1 \le i \le k,
\end{equation}
and
\begin{equation}\label{def-xi-2}
\frac{d}{d t} x_i(t, s, \xi) = - \lambda_i \big(x_i(t, s, \xi) \big) \mbox{ and }  x_i(s, s, \xi) = \xi \mbox{ if } k+1 \le i \le k+ m.
\end{equation}

\medskip
The following definition of broad solutions for \eqref{sys-v} is used in this paper

\begin{definition} \label{def-broad}A function $v=(v_1,\ldots, v_{k+m}): (0, T) \times (0, 1) \to \mathbb{R}^{k+m}$  is called a broad solution  of \eqref{sys-v} if $v\in [L^\infty \big((0, T) \times (0, 1) \big)]^{k+m} \cap [C\big([0, T]; L^2(0, 1) \big)]^{k+m} \cap [C\big([0, 1]; L^2(0, T)\big)]^{k+m}$ and if, for almost every $(\tau, \xi) \in (0, T) \times (0, 1)$, we have

\medskip
1. For $k+1 \le i \le k+ m$,
\begin{align}\label{broad-1}
v_i(\tau, \xi) = & \int_t^\tau \sum_{j=1}^n \Big( C_{i j} \big(x_i(s, \tau, \xi) \big) v_j \big(s, x_i(s, \tau, \xi)\big) + D_{i j}\big(x_i(s, \tau, \xi) \big) v_j(s, 0)  +  f_i\big(s, x_i(s, \tau, \xi) \big) \Big) \, ds \nonumber \\[6pt]
&  + \sum_{r=1}^R \sum_{j =1}^n A_{r, ij}(t) v_j\big(t, x_r  \big) + \int_{0}^1 \sum_{j =1}^n M_{ij}(t, x) v_j(t, x) \, dx + h(t),
\end{align}
if  $x_i(0, \tau, \xi) > 1$ and $t$ is such that $x_i(t, \tau, \xi) = 1$, and
\begin{align}\label{broad-2}
v_i(\tau, \xi) = &  \int_0^\tau \sum_{j=1}^n \Big(  C_{ij} \big(x_i(s, \tau, \xi)\big) v_j\big(s, x_i(s, \tau, \xi)\big) + D_{i j}\big(x_i(s, \tau, \xi)\big) v_j(s, 0) \nonumber \\[6pt]
& +  f_i\big(s, x_i(s, \tau, \xi) \big) \Big) \, ds   + v_{0, i} \big(x_i(0, \tau, \xi) \big),  \end{align}
if $x_i(0, \tau, \xi) < 1$.

\medskip
2. For $1 \le i \le k$,
\begin{align}\label{broad-3}
v_i(\tau, \xi) = & \int_t^\tau \sum_{j=1}^n \Big( C_{ij} (x_i(s, \tau, \xi)) v_j(s, x_i(s, \tau, \xi)) +   D_{ij}(x_i(s, \tau, \xi)) v_j(s, 0)  \nonumber \\[6pt] & \qquad +  f_i\big(s, x_i(s, \tau, \xi) \big) \Big)\, ds
 +  \sum_{j=1}^{m} B_{ij} v_{j+k}(t, 0) + g_i(t),
\end{align}
if $x_i(0, \tau, \xi) < 0$ and $t$ is such that $ x_i(t, \tau, \xi) = 0$ where $v_{j+k}(t, 0)$ is defined  by the RHS of \eqref{broad-1} or \eqref{broad-2} with $(\tau, \xi) = (t, 0)$, and
\begin{align}\label{broad-4}
v_i(\tau, \xi) = & \int_0^\tau \sum_{j=1}^n \Big(  C_{ij} (x_i(s, \tau, \xi)) v_j(s, x_i(s, \tau, \xi)) + D_{ij}(x_i(s, \tau, \xi)) v_j(s, 0)  \nonumber \\[6pt] & \qquad   +  f_i\big(s, x_i(s, \tau, \xi) \big) \Big) \, ds   + v_{0, i} \big(x_i(0, \tau, \xi) \big),
\end{align}
if $x_i(0, \tau, \xi) >0$.
\end{definition}

Here and in what follows,  $v_i$ denotes the $i$-th component of $v$, $v_{i, 0}$ denotes the $i$-th component of   $v_0$, and $A_{r,ij}$ denotes the $(i, j)$ component of $A_r$.

\newtextb{Classical solutions are smooth broad solutions. Conversely, smooth broad solutions are classical solutions. This is a consequence of the following lemma on the  well-posedness of \eqref{sys-v}:}

\begin{lemma} \label{lem-WP} Let $v_0 \in [L^\infty (0, 1)]^n$, $f \in \big[L^\infty \big((0, T) \times (0, 1) \big) \big]^n$, $g \in [L^\infty(0, T)]^k$, and \\
$h \in [L^\infty(0, T)]^m$,  and assume \eqref{assumption-1} and \eqref{assumption-2}. Then \eqref{sys-v} has a unique broad solution $v$.
\end{lemma}

\begin{proof} The proof is based on a fixed point argument.
To this end, define ${\cal F}$ from  ${\cal Y}: = \big[L^\infty\big((0, T) \times (0, 1) \big)\big]^n \cap \big[C\big([0, T]; L^2(0, 1)\big)\big]^n \cap \big[C\big([0, 1]; L^2(0, T) \big)\big]^n$ into itself
as follows, for $v \in {\cal Y} $ and for $(\tau, \xi) \in (0, T) \times (0, 1)$,
\begin{multline}
\big( {\cal F}(v) \big)_i(\tau, \xi) \mbox{ is the RHS of \eqref{broad-1} or  \eqref{broad-2} or  \eqref{broad-3} or  \eqref{broad-4}} \\
\mbox{ under the corresponding conditions.}
\end{multline}
Set
$$
{\cal N} : = \|B \|_{L^\infty} + \|C \|_{L^\infty} + \| D\|_{L^\infty} + \| M\|_{L^\infty} + \sum_{r=1}^R \| A_{r}\|_{L^\infty}.
$$
We claim that there exist two constants $L_1, \, L_2 > 1$ depending only on $c$, ${\cal N}$, and $\Sigma$ such that  ${\cal F}$ is a contraction map for  the norm
\begin{equation}\label{def-norm}
\| v \| := \sup_{1 \le i \le n} \mbox{ ess sup }_{(\tau, \xi) \in (0, T) \times (0, 1)} e^{-L_1 \tau - L_2 \xi} |v_i(\tau, \xi)|.
\end{equation}

We first consider the case where $\big({\cal F}(v) \big)_i(\tau, \xi)$ is given by the RHS of \eqref{broad-2} or \eqref{broad-4}. We claim that, for $v, \hat v \in {\cal Y}$,
\begin{equation}\label{board-part1}
e^{-L_1 \tau - L_2 \xi} \big| \big({\cal F}(v) \big)_i(\tau, \xi) - \big({\cal F}( \hat v)\big)_i (\tau, \xi) \big| \le \|v - \hat v \|/(10n),
\end{equation}
if  $L_2$ is large enough and  $L_1$ is much larger than $L_2$. Indeed, we have, with $V = v - \hat v$,
\begin{align*}
\big| \big({\cal F}(v) \big)_i(\tau, \xi) - \big({\cal F}( \hat v)\big)_i (\tau, \xi) \big| \le &  {\cal N}  \int_0^\tau  \big(  |V(s, x_i(s, \tau, \xi))| + |V(s, 0)| \big) \, ds \\[6pt]
\le &  2\sqrt{n}  {\cal N}L_1^{-1} \| V \| e^{\tau L_1 + L_2},
\end{align*}
which implies \eqref{board-part1}.

We next consider the  case where $\big({\cal F}(v) \big)_i(\tau, \xi)$ is given by the RHS of \eqref{broad-1}. We have
\begin{align*}
\big|\big({\cal F}(v) - {\cal F}(\hat v)   \big)_i(\tau, \xi) \big| \le  & {\cal N} \left( \int_t^\tau \big(|V(s, x_i(s, \tau, \xi))| + |V(s, 0)| \big) \, ds +  |V(t, x_r)|  + \int_{0}^1   |V(t, x)| \, dx \right) \\[6pt]
\le &  2 \sqrt{n} {\cal N}   \big( L_1^{-1} e^{L_1 \tau + L_2} \| V\|   +  e^{L_1 t + L_2 c} \| V \| +  L_2^{-1} e^{L_1 t + L_2} \|V \| \big).
 \\[6pt] \le &    4 \sqrt{n} {\cal N} \big( L_1^{-1} e^{L_1 \tau + L_2} \| V\|   +  L_2^{-1} e^{L_1 t + L_2} \| V \| \big),
\end{align*}
if $L_2$ is large enough since $c  < 1$.  Since $\tau - t \ge C (1 - \xi)$  for some positive constant depending only on $\Sigma$, $k$, and $m$ by the definitions of $x_i$ and $t$, it follows that
\begin{equation}\label{board-part2}
e^{-L_1 \tau - L_2 \xi} \big| \big({\cal F}(v) - {\cal F} (\hat v) \big)_i(\tau, \xi) \big| \le \|V \|/2,
\end{equation}
if $L_2$ is large and $L_1$ is much larger than $L_2$.


We finally consider the  case where $\big({\cal F}(v) \big)_i(\tau, \xi)$ is given by the RHS of \eqref{broad-3}. We have
\begin{align}\label{board-part3-1}
\big| \big({\cal F}(v) - {\cal F}(\hat v)  \big)_i(\tau, \xi)| & \le   {\cal N}   \left( \int_t^\tau \big( |V(s, x_i(s, \tau, \xi))| +   |V(s, 0)| \big)\, ds  +   \sum_{j=k+1}^{k+m} | V_{j}(t, 0)|  \right)  \nonumber \\[6pt]
& \le  2 \sqrt{n}{\cal N}   \Big( L_1^{-1} e^{L_1 \tau + L_2} \| V \|  +   \sum_{j=k+1}^{k+m} | V_{j}(t, 0)|  \Big).
\end{align}
From \eqref{broad-1} and \eqref{broad-2}, as in the previous cases, we have
\begin{equation*}
{\cal N} e^{L_2}e^{-L_1 t} | V_{j}(t, 0)| \le \|V \| / (10n) \quad \mbox{ for } k+1 \le j \le k+m,
\end{equation*}
if $L_2$ is large and $L_1$ is much larger than $L_2$.  We  derive from \eqref{board-part3-1} that
\begin{equation}\label{board-part3}
e^{-L_1 \tau - L_2 \xi}\big| \big({\cal F}(v) \big)_i(\tau, \xi) - \big({\cal F}(\hat v)\big)_i (\tau, \xi) \big| \le \|V\|/2,
\end{equation}
if $L_2$ is large enough and $L_1$ is much larger than $L_2$.

Combining \eqref{board-part1}, \eqref{board-part2}, and \eqref{board-part3} yields, for $v, \hat v \in {\cal Y}$,
\begin{equation*}
\|{\cal F}(v) - {\cal F}(\hat v) \| \le \|v - \hat v \|/ 2.
\end{equation*}
Thus ${\cal F}$ is a contraction mapping. By the Banach fixed-point theorem, there exists a unique $v \in {\cal Y}$ such that
$$
{\cal F}(v) = v.
$$
The proof is complete.
\end{proof}

Concerning $K$, we have the following result:

\begin{lemma} \label{lem2}  Assume \eqref{assumption-1} and \eqref{assumption-2}. There exists a unique broad bounded solution $K: \D \to \mR^{n \times n}$ of system \eqref{equation-K}, \eqref{bdry-K-1}, \eqref{bdry-K-2}, and \eqref{bdry-K-3}.  Moreover, $(\gamma, B)\in \mR  \times \mR^{k \times m}  \mapsto  K \in [L^\infty(\D)]^{n \times n}$ is analytic.
\end{lemma}

\begin{remark} \rm
The broad solution meaning of $K$ is understood via the characteristic approach similar to Definition~\ref{def-broad}. The continuity assumptions in Definition~\ref{def-broad} are replaced by the assumption that $\tilde K(\cdot, y) \in L^2([0, 1])$ is continuous w.r.t. to $y \in [0, 1)$ where $\tilde K(x, y) = K((1-y)x, y)$ and similar facts for $x$ and $x+y$ variables.
\end{remark}

\begin{proof} Using similar approach, one can establish the existence and uniqueness of $K$. The real analytic with respect to each component of $B$ can be proved by  showing that $K$ is holomorphic with respect to each component of $B$.  In fact, for notational ease, assuming again  that $\gamma =1$,  one can prove that
$$
\frac{\partial K}{\partial B_{pq}} = \hat K \mbox{ in } \D
$$
(the derivative is understood for a complex variable),  where $\hat K$ is the bounded broad solution of \eqref{equation-K},
\begin{equation*}
\hat K_{ij} (x,x)= 0 \mbox{ for } x \in (0, 1), \,  1 \le i  \neq j \le k+m,
\end{equation*}
\begin{equation*}
\hat K_{ij}(1, y) = 0 \mbox{ for } y \in (0, 1), \, 1 \le i < j \le k \mbox{ or } k+1 \le j < i \le k+ m,
\end{equation*}
(which are derived from \eqref{bdry-K-1},  and \eqref{bdry-K-3}) and for $(i, j ) \in {\cal J}$,
\begin{equation}\label{bdry-K-2-hatK}
\hat K_{i j}(x, 0) = \sum_{(r, s) \not \in {\cal J}} c_{i j r s}(B) \hat K_{r s}(x, 0) +  \sum_{(r, s) \not \in {\cal J}} \frac{\partial c_{i j r s}(B)}{\partial B_{pq}} K_{r s}(x, 0)  \mbox{ for } x \in (0, 1),
\end{equation}
which is obtained from \eqref{bdry-K-2}. The existence and uniqueness of $\hat K$ can be established as in the proof of Lemma~\ref{lem-WP} where the second term in the RHS of  \eqref{bdry-K-2-hatK} plays a role as the one of $g$ in Lemma~\ref{lem-WP}.  The details  of the proof are left to the reader.

The analyticity with respect to $\gamma$ can be proved by  showing that $K$ is holomorphic with respect to $\gamma$. In fact, one can prove that
$$
\frac{\partial K}{\partial \gamma} = \hbK \mbox{ in } \D,
$$
(the derivative is understood for a complex variable) where  $\hat {\bf K}$ is the bounded broad solution of
\begin{equation}\label{equation-K-bf}
\partial_y \hbK (x, y) \Sigma (y) +  \Sigma(x) \partial_x \hbK(x, y) +  \hbK(x, y)  \Sigma'(y) - \gamma \hbK(x, y) C(y)= K(x, y) C(y) \mbox{ in } \D,
\end{equation}
by \eqref{equation-K},
\begin{equation}\label{bdry-K-1-bf}
\hbK_{ij} (x,x) = \frac{C_{ij}(x)}{ a_{ij}(x) - b_{ij}(x)}  \mbox{ for } x \in (0, 1), \,  1 \le i  \neq j \le k+m,
\end{equation}
\begin{equation*}
\hbK_{ij}(1, x) = 0 \mbox{ for } x \in (0, 1), \, 1 \le i < j \le k \mbox{ or } k+1 \le j < i \le k+ m,
\end{equation*}
by \eqref{bdry-K-1},  and \eqref{bdry-K-3},
and
\begin{equation}\label{bdry-K-2-bf}
\hbK_{i j}(x, 0) = \sum_{(r, s) \not \in {\cal J}} c_{i j r s}(B)\hbK_{r s}(x, 0) \mbox{ for } x \in (0, 1),   \, (i, j) \in {\cal J}
\end{equation}
by  \eqref{bdry-K-2}.  Here $K(x, y)$ denotes the solution corresponding to fixed $\gamma$ and $B$.  Note that $\gamma$ does not appear in the boundary conditions of $\hbK$. The details are omitted.
\end{proof}

A connection between $w$ and $u$ is given in the following proposition.

\begin{proposition} \label{pro-eq-u} Let $w_0 \in L^\infty\big( (0, 1)\big)$ and let $w \in L^\infty  \big((0, T) \times (0, 1) \big)$. Define $u_0$ and  $u$ from $w_0$ and $w$ by \eqref{backstepping} respectively  and let  $S$ be given by \eqref{def-S-*}.
Assume \eqref{assumption-1} and \eqref{assumption-2}.  We have,
if $w$ is a   broad solution of the system
\begin{equation}\label{Sys-w-1}
\left\{\begin{array}{cl}
\partial_t w (t, x) =  \Sigma(x) \partial_x w (t, x) + C(x) w(t, x)  &  \mbox{ for } (t, x)  \in (0, T) \times (0, 1),\\[6pt]
w_-(t, x = 0) = B w_+(t, x = 0) & \mbox{ for } t \in (0, T), \\[6pt]
u_+(t, 1) = \dsp \sum_{r=1}^R A_r (t) u\big(t, x_r \big) +   \int_{0}^1 M(t, y) u(t, y) \, dy &  \mbox{ for }  t \in (0, T), \\[6pt]
w(t = 0)(x) = w_0(x) & \mbox{ for } x \in (0, 1),
\end{array} \right.
\end{equation}
then $u$ is a   broad  solution of the system
\begin{equation}\label{Sys-u-1}
\left\{\begin{array}{cl}
\partial_t u (t, x) =  \Sigma(x) \partial_x u (t, x) + S(x) u(t, 0)  &  \mbox{ for } (t, x)  \in  (0, T) \times (0, 1),\\[6pt]
u_-(t, x = 0) = B u_+(t, x = 0) & \mbox{ for } t \in (0, T),  \\[6pt]
u_+(t, 1) = \dsp \sum_{r=1}^R A_r (t) u\big(t, x_r  \big) +   \int_{0}^1 M(t, y) u(t, y) \, dy &  \mbox{ for }  t \in (0, T), \\[6pt]
u(t = 0, x) = u_0(x) & \mbox{ for } x \in (0, 1).
\end{array} \right.
\end{equation}
\end{proposition}

\begin{remark} \rm \newtextb{In the two sides of the third condition in \eqref{Sys-w-1}, $u$ is given by \eqref{backstepping}. Therefore, this condition is understood as a condition on $w$.}
By Lemma~\ref{lem-WP}, there exist a unique  broad solution $w$ of \eqref{Sys-w-1} and a unique  broad solution $u$ of \eqref{Sys-u-1}.
\end{remark}

\begin{proof} We first assume in addition that $C$ and $\Sigma$ are smooth on $[0, 1]$. Let $K_{n}$ be a $C^1$- solution of \eqref{equation-K} and \eqref{bdry-K-1} such that
\begin{equation}\label{pro-Kn}
\|K_n\|_{L^\infty(\D)} \le M  \quad \mbox{ and } \quad K_n \to K \mbox{ in } L^1(\D),
\end{equation}
where $M$ is a positive constant independent of $n$.
Such a $K_n$ can be obtained  by considering the solution of \eqref{equation-K}, \eqref{bdry-K-1}, and
\begin{equation*}
K_{n, i j}(x, 0) = \sum_{(r, s) \not \in {\cal J}} c_{i j r s}(B)K_{n, r s}(x, 0) + g_n(x, 0) \mbox{ for } x \in (0, 1)
\end{equation*}
and
\begin{equation*}
K_{n, ij}(1, x) = h_n(x) \mbox{ for } x \in (0, 1)
\end{equation*}
instead of \eqref{bdry-K-2} and \eqref{bdry-K-3} respectively where $(g_n)$, $(h_n)$
are chosen such that $(g_n)$ and $(h_n)$ are bounded in $L^\infty(0, 1)$, $(g_n), (h_n) \to 0$ in $L^1(0, 1)$,  and the compatibility conditions hold for $K_n$ at $(0, 0)$ and $(1, 1)$ \footnote{One needs to establish the stability for $L^1$-norm for the system of $K$. This can be done as in \cite{Bressan}.}. Set, for sufficiently small positive $\eps$,
$$
w_\eps(t, x) = \frac{1}{2 \eps} \int_{t-\eps}^{t + \eps} w(s, x) \, ds \mbox{ in } (\eps, T -\eps) \times (0, 1).
$$
Then $w_\eps \in W^{1, \infty} \big( (\eps, T -\eps)  \times (0, 1) \big)$ and
$$
\partial_t w_\eps (t, x) = \Sigma(x) \partial_x w_\eps(t, x) +  C(x) w_\eps (t, x)  \mbox{ in } (\eps, T -\eps)  \times (0, 1).
$$
Define
$$
u_{n, \eps} (t, x) = w_\eps (t, x) - \int_0^x K_n(x, y) w_\eps (t, y) \, dy  \mbox{ in } (\eps, T -\eps)  \times (0, 1)
$$
and
$$
u_n(t, x) = w(t, x) - \int_0^x K_n(x, y) w(t, y) \, dy  \mbox{ in } (0,  T)  \times (0, 1).
$$
As in \eqref{eq-u}, we have
\begin{equation*}
\partial_t u_{n, \eps}(t, x) = \Sigma (x) \partial_x u_{n, \eps}(t, x) + K_n(x, 0) \Sigma(0) u_{n, \eps} (t, 0) \mbox{ in }   (\eps, T - \eps) \times (0, 1).
\end{equation*}
By letting $\eps \to 0$, we obtain
\begin{equation*}
\partial_t u_{n}(t, x) = \Sigma (x) \partial_x u_{n}  (t, x) + K_n(x, 0) \Sigma(0) u_{n} (t, 0) \mbox{ for } (t, x)  \in (0, T ) \times (0, 1).
\end{equation*}
By letting $n \to + \infty$ and using \eqref{pro-Kn}, we derive that
\begin{equation*}
\partial_t u (t, x) = \Sigma (x) \partial_x u(t, x)  + K(x, 0) \Sigma(0) u (t, 0) \mbox{ for } (t, x)  \in (0, T ) \times (0, 1).
\end{equation*}
This yields the first equation of \eqref{Sys-u-1}. The other parts of \eqref{Sys-u-1} are clear from the definition of $w_0$ and $w$.

We next consider the general case, in which no further additional smooth assumption on  $\Sigma$ and $C$ is required. The proof in the case
can be derived from the previous case by approximating  $\Sigma$ and $C$ by smooth functions. The details are omitted.
\end{proof}

The fact that $w$ is  uniquely determined from $u$ is a consequence of the following standard result on the Volterra equation of the second kind whose proof is omitted; this implies in particular that $w(t, \cdot) \equiv 0$ in $(0, 1)$ if $u(t, \cdot) \equiv 0$ in $(0, 1)$.

\begin{lemma}\label{lem-Volterra} Let $d  \in \mN$,  $ \tau_1, \tau_2 \in \mR$ be such that $\tau_1 < \tau_2$ and let $G: \big\{ (t, s): \tau_1 \le s \le t \le \tau_2 \big\} \to \mR^{d \times d}$ be bounded measurable. For every $F \in \big[L^\infty(\tau_1, \tau_2)\big]^{d}$,  there exists a unique solution $U \in \big[L^\infty(\tau_1, \tau_2) \big]^{d}$ of the following equation
\begin{equation*}
U(t) = F(t) + \int_{\tau_1}^t G(t, s) U(s) \, ds \quad \mbox{ for } t \in (\tau_1, \tau_2).
\end{equation*}
\end{lemma}



\section{Null-controllability for generic $\gamma$ and $B$ - Proof of Theorem~\ref{thm1}} \label{sect-thm1}

\subsection{Proof of part 1) of Theorem~\ref{thm1}} \label{sect-thm1-part1}

Choose $u_{k+1}(t, 1) = 0$ for $t \ge 0$. Since $S_{++} = 0$ by  \eqref{cond-S},  we have
$$
u_{k+1}(t, 0) = 0 \mbox{ for } t \ge \tau_{k+1} \quad \mbox{ and } \quad u_{k+1}(T_{opt}, x) = 0 \mbox{ for } x \in (0, 1).
$$
This implies, by  \eqref{bdry:x=0-w},
$$
u_{i}(t, 0) = 0 \mbox{ for } t \ge \tau_{k+1}, 1 \le i \le k.
$$
We derive from \eqref{eq-u} that
$$
u_i(T_{opt}, x) = 0 \mbox{ for } x \in (0, 1), 1 \le i \le k.
$$
The null-controllability at the time $T_{opt}$ is attained for $u$ and hence for $w$ by Lemma~\ref{lem-Volterra}.

\subsection{Proof of part 3) of  Theorem~\ref{thm1}} \label{sect-thm1-part2}

We here establish  part 3) of Theorem~\ref{thm1} even for  $m \ge 1$.  We hence assume that $m \ge 1$ in this section.  Set
\begin{equation}\label{def-tl}
t_0 = T_{opt}, \quad t_1 = t_0 - \tau_1 , \cdots, \quad t_k = t_0 -  \tau_k,
\end{equation}
and, for $1 \le l \le k$,
\begin{equation}\label{def-xij}
x_{0, l} = 0 \quad \mbox{ and } \quad x_{i, l} = x_l(t_0, t_i, 0),  \quad \mbox{ for } 1 \le i \le l.
\end{equation}
Recall that $x_l$ is defined in \eqref{def-xi-1} for $1 \le l \le k$. (See Figure~\ref{fig-xt} in the case where $\Sigma$ is constant.)

\medskip
In the next two sections, we deal with the case $m \ge k$ and $m < k$ respectively.

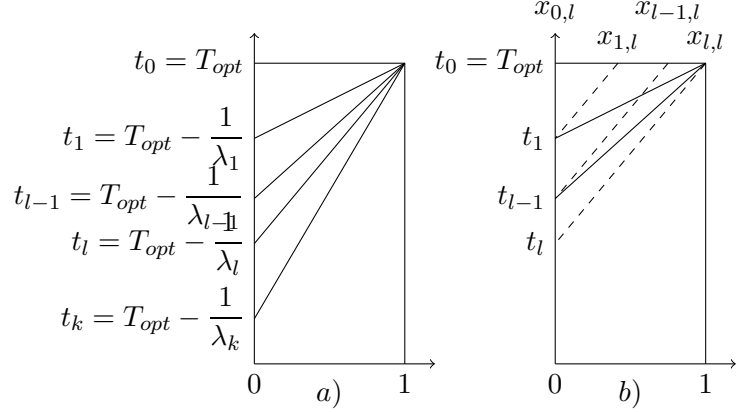
\begin{figure}
\centering
\begin{tikzpicture}[scale=2]

\draw[->] (0,0) -- (1.2,0);
\draw[->] (0,0) -- (0,2.2);
\draw[] (1,0) -- (1,2.0);
\draw[] (0,2) -- (1,2);
\draw[] (1,2) -- (0,1.5);
\draw[] (1,2) -- (0,1.1);
\draw[] (1,2) -- (0,0.8);
\draw[] (1,2) -- (0,0.3);

\draw (0, 2) node[left]{$t_0= T_{opt}$};
\draw (0, 1.5) node[left]{$\dsp t_1= T_{opt} - \frac{1}{\lambda_{1}}$};
\draw (0, 1.1) node[left]{$\dsp t_{l-1}= T_{opt} - \frac{1}{\lambda_{l-1}}$};
\draw (0, 0.8) node[left]{$\dsp t_l= T_{opt} - \frac{1}{\lambda_{l}}$};
\draw (0, 0.3) node[left]{$\dsp t_k= T_{opt} - \frac{1}{\lambda_{k}}$};

\draw (0.0, 0.0) node[below]{$0$};
\draw (1.0, 0.0) node[below]{$1$};

\draw (0.5, -0.2) node[]{$a)$};


\newcommand\z{2.0}

\draw[->] (0+\z,0) -- (1.2+\z,0);
\draw[->] (0+\z,0) -- (0+\z,2.2);
\draw[] (1+\z,0) -- (1+\z,2.0);
\draw[] (0+\z,2) -- (1+\z,2);
\draw[] (1+\z,2) -- (0+\z,1.5);
\draw[] (1+\z,2) -- (0+\z,1.1);

\draw[dashed] (1+\z,2) -- (0+\z,0.8);
\draw[dashed] (0.5/1.2 +\z,2) -- (0+\z,1.5);
\draw[dashed] (0.9/1.2+\z,2) -- (0+\z,1.1);

\draw (0+\z, 2) node[left]{$t_0= T_{opt}$};
\draw (0+\z, 1.5) node[left]{$\dsp t_1$};
\draw (0+\z, 1.1) node[left]{$\dsp t_{l-1}$};
\draw (0+\z, 0.8) node[left]{$\dsp t_l$};

\draw (0+\z,2.2) node[above]{$x_{0, l}$};
\draw (0.5/1.2+\z,2) node[above]{$x_{1, l}$};
\draw (0.9/1.2+\z,2.2) node[above]{$x_{l-1, l}$};
\draw (1+\z, 2) node[above]{$x_{l, l}$};

\draw (0.0+\z, 0.0) node[below]{$0$};

\draw (1.0+\z, 0.0) node[below]{$1$};

\draw (0.5+\z, -0.2) node[]{$b)$};

\end{tikzpicture}
\caption{The definition of $t_l$ is given in a) and the definition of $x_{i, l}$ is given in $b)$ where dashed lines have the same slope for constant $\Sigma$.}\label{fig-xt}
\end{figure}


\subsubsection{On the case $m \ge k$} \label{sect-m>=k}

\newtextb{The idea of the proof is to derive  sufficient conditions  to be able to steer the control system from the initial data to 0 at the time $T_{opt}$. These conditions will be written under the form $U + {\cal K} U = {\cal F}$ (see \eqref{equation-U}) where ${\cal K}$ is an analytic, compact operator with respect to $\lambda$ and $F$ depending on the initial data. We then apply the Fredholm theory to obtain the conclusion. We now proceed the proof.}

We begin with  deriving  conditions for controls to reach the null-controllability at the time $T_{opt}$.   First,  if $m > k$, choose the control
\begin{equation}\label{preparation-u-feedback-0}
\mbox{$u_{l}(t, 1) = 0$ for $0 \le t \le T_{opt} -  \tau_l$  and $k + 1 \le l \le m$}.
\end{equation}
Note that in the case $T_{opt} = \tau_l$, one does not impose any condition for $u_{l}$ in \eqref{preparation-u-feedback-0}.
 Second, choose the control, for $1 \le i \le k$,
\begin{equation}\label{preparation-u-feedback}
u_{m+i}(t, 1) = 0  \mbox{ for } 0 \le t < T_{opt} - \tau_i - \tau_{m+i}.
\end{equation}
Note that in the case $T_{opt} = \tau_{m+i}  + \tau_i$, one does not impose any condition for $u_{m+i}$ in \eqref{preparation-u-feedback}.

Requiring \eqref{preparation-u-feedback-0} and \eqref{preparation-u-feedback} is just a preparation step, other choices are possible. The main part in the construction of the controls  is to  choose the control  $u_{m+i}(t, 1)$ for $t \in (T_{opt} -\tau_{i} - \tau_{m+i}, T_{opt} -\tau_{i})$ and  for $1 \le i \le k$ such that the following $k$ conditions hold:

\medskip
\noindent $a_1)$
$$
u_{k}(T_{opt}, x) = 0 \mbox{ for }  x \in (x_{0, k}, x_{1, k}), \quad
\cdots, \quad  u_{1}(T_{opt}, x) = 0 \mbox{ for }  x \in (x_{0, 1}, x_{1, 1}).
$$

\medskip
\noindent $a_2)$
$$
u_{k}(T_{opt}, x) = 0 \mbox{ for } x \in  (x_{1, k}, x_{2, k}), \quad
\cdots, \quad  u_{2}(T_{opt}, x) = 0 \mbox{ for }  x \in (x_{1, 2}, x_{2, 2}).
$$

\dots

\medskip
\noindent $a_k)$
$$
u_{k}(T_{opt}, x) = 0 \mbox{ for } x \in  (x_{k-1, k}, x_{k, k}).
$$

Set
\begin{equation}\label{def-X}
{\cal X}: = L^2(t_1, t_0) \times L^2(t_2, t_0) \times \cdots \times L^2(t_k, t_0)
\end{equation}
and  denote
$$
U_j(t) = \big(u_{m+j} (t, 0), \cdots, u_{m+k} (t, 0) \big)\tr   \mbox{ for } 1 \le j \le k
$$
and
$$
V_j (t) = \big( u_{k+1}(t, 0), \dots, u_{j}(t, 0)\big)\tr  \mbox{ for } m \le j \le m + k,
$$

We determine
$$
\big(u_{m+1}(\cdot, 0), \cdots, u_{m+k}(\cdot, 0) \big)\tr  \in {\cal X}.
$$
via the conditions in $a_1)$, $a_2)$, \dots, $a_k)$. \medskip
\newtextb{Let us now find necessary and sufficient conditions on $
\big(u_{m+1}(\cdot, 0), \cdots, u_{m+k}(\cdot, 0) \big)\tr  \in {\cal X}
$ so that $a_1)$, \dots, $a_k)$ hold. These are analysed 
in $b_1)$, \dots, $b_k)$ below  respectively.}

\medskip

\noindent $b_1)$ \newtextb{From  \eqref{eq-u} and \eqref{def-S},  using} the characteristic method and the fact that $S_{ij} = 0$ for $ 1\le i, j \le k$,  one can write the conditions in $a_1)$ under the form
$$
(u_1, \cdots, u_k)\tr (t, 0) + \int_t^{t_0} {\cal L}_{1}(t, s) (u_{k+1}, \cdots, u_{k +m})\tr (s, 0)  \, ds = 0 \quad  \mbox{ for } t_1 \le t \le t_0,
$$
for some ${\cal L}_1 \in \big[L^\infty \big( (t, s) ; t_1 \le t \le s \le t_0\big) \big]^{k \times m}$.
Using \eqref{cond-B-1} with $i = k$ \newtext{provided $m>k$}, one can write the above equation  under the  form
\begin{equation}\label{cond-a1}
U_1(t) = A_1  V_m(t) + \int_t^{t_0} G_{1}(t, s) V_m (s) \, ds  + \int_t^{t_0} H_{1}(t, s) U_1(s) \, ds \quad \mbox{ for } t_1 \le t \le t_0,
\end{equation}
for some $G_1 \in [L^\infty (\{ (t, s) ; \, t_1 \le t \le s \le t_0 \})]^{k \times (m-k)}$ and $H_1  \in [L^\infty (\{ (t, s) ; \, t_1 \le t \le s \le t_0 \})]^{k \times k} $ depending only on $S$, $B$, and $\Sigma$, and some matrix $A_1 \in \mR^{k \times (m-k)}$ depending only on $B$. \newtext{In the case $m = k$, one chooses $H_1 = 0$ (there are not $A_1$ and $G_1$ in this case by the convention)}
Since $K$ is analytic with respect to $(\gamma, B)$,  one can check that ${\cal L}_1$ is analytic with respect to $(\gamma, B)$. \newtextb{In fact, ${\cal L}_1$ depends linearly on $S$  and so analytically on $(\gamma, B)$,  and if $\gamma =0$ then ${\cal L}_1 =0$.} This implies that $G_{1}$ and $H_1$ are analytic with respect to $(\gamma, B)$.
It is also clear that $A_1$ is  analytic with respect to $(\gamma, B)$ as well.

\begin{remark} \label{rem-m=k} \rm \newtext{ In the case $m = k$,  $U_1(t) = 0$ for $t_1 \le t \le t_0$. This fact will be used to deal with the case $m  < k$}.
\end{remark}

\noindent $b_2)$
Similar to \eqref{cond-a1}, the conditions in $a_2)$ \newtextb{is equivalent to}
\begin{equation}\label{cond-a2}
U_{2}(t) = A_{2} V_{m+1}(t) +  \int_t^{t_0} G_{2}(t, s) V_{m+1}(s) \, ds  + \int_t^{t_0} H_{2}(t, s) U_{2}(s) \, ds \quad \mbox{ for } t_2 \le t < t_1,
\end{equation}
for some constant matrix $A_2$ and some bounded functions $G_{2}$ and $H_{2}$ defined in   $ \big\{(s, t); t_2 \le t \le s \le t_0 \big\}$ which depend only on $S$, $B$, and $\Sigma$.  Moreover, $A_2$, $G_2$ and $H_2$ are analytic with respect to $(\gamma, B)$.

\dots

\medskip
\noindent $b_k)$ Similar to \eqref{cond-a1}, the condition in $a_k)$ \newtextb{is equivalent to}
\begin{equation}\label{cond-ak}
U_k(t) =  A_{k} V_{m+k-1}(t)  + \int_t^{t_0} G_{k}(t, s) V_{m+k - 1}(s) \, ds  + \int_t^{t_0} H_{k}(t, s) U_k(s) \, ds \quad \mbox{ for } t_k \le t < t_{k-1},
\end{equation}
for some constant matrix $A_k$ and some bounded functions $G_{k}$ and $H_{k}$ defined in $ \big\{(s, t); t_k \le t \le s \le t_0 \big\}$ which depends only on $S$, $B$, and $\Sigma$.  Moreover, $A_k$, $G_k$ and $H_k$ are analytic with respect to $(\gamma, B)$.

We are next concerned about the relations between the components of $u(t,0)$. We have, by the property  of $S_{++}$ in \eqref{cond-S} and the form of $S$ in \eqref{def-S},  and \eqref{preparation-u-feedback-0} and  \eqref{preparation-u-feedback},
\begin{equation}\label{cond-ck}
u_{m+k}(s, 0) = F_{m+k}(s) \mbox{ for } 0 \le s \le t_k,
\end{equation}
\begin{equation}\label{cond-ck-1}
u_{m+ k -1}(s, 0) = F_{m+ k-1}(s) +  \int_{0}^s {\cal G}_{m+ k-1, m + k} (\xi) u_{m + k}(\xi, 0) \, d \xi \mbox{ for } 0 \le s \le t_{k-1},
\end{equation}
\begin{align}\label{cond-ck-2}
u_{m + k -2}(s, 0) = &  F_{m + k-2}(s) +  \int_{0}^s {\cal G}_{m+ k-2, m + k}(\xi) u_{m + k}(\xi, 0) \, d \xi \nonumber \\[6pt]
 & + \int_{0}^s {\cal G}_{m + k-2, m + k-1}(\xi) u_{m + k - 1}(\xi, 0) \, d \xi \mbox{ for } 0 \le s \le t_{k-2},
\end{align}
\dots
\begin{equation}\label{cond-c1}
u_{k+1}(s, 0) = F_{k+1}(s) +  \int_{0}^s \sum_{j = k+2}^{k  + m } {\cal G}_{k+1, j} (\xi) u_j(\xi, 0) \, d \xi  \mbox{ for } 0 \le s \le t_1,
\end{equation}
where ${\cal G}_{i, j}$ depends only on $S$ and $\Sigma$ and is analytic with respect to $(\gamma, B)$,  and $F_i$ depends  only on the initial data. Here we also use  \eqref{preparation-u-feedback-0} and \eqref{preparation-u-feedback}.


Using  (\ref{cond-ck}-\ref{cond-c1}), one can write the equations in $b_1)$, \dots, $b_k)$ under the form
\begin{equation}\label{equation-U}
U + {\cal K} (U) = F \mbox{ in } {\cal X}.
\end{equation}
where
$$
U = \big(u_{\newtextb{m+1}}(\cdot, 0), \cdots, u_{m+k}(\cdot, 0) \big)\tr
$$
and ${\cal K}$ \newtextb{is a Hilbert-Schmidt operator, therefore a compact operator,} and it is analytic with respect to $(\gamma, B)$.

By the theory of analytic compact theory (see, e.g.,  \cite[Theorem 8.92]{ReRo}), for each $B \in {\cal B}$,
$I + {\cal K}$ is  invertible outside a discrete set of $\gamma $ in $\mR$ since $\|{\cal K} \|$ is small if $\gamma$ is small.

Using this fact, since ${\cal B}$ has a finite number of connected components, there exists a discrete subset of $\mR$ such that outside this set, $I + {\cal K}$ is invertible for almost every $B \in {\cal B}$ by the Fredholm theory for analytic compact operator.

Consider $(\gamma, B)$ such that $I + {\cal K}$ is invertible. Then   equation \eqref{equation-U} has a unique solution for all  $F$ in ${\cal X}$.  One can check that if $F$ is bounded then $U$ is bounded since ${\cal K} U$ is bounded.
To obtain the null-controllability at the time $T_{opt}$,  \newtextb{in addition to the preparation step, one chooses $u_{k+m}(1, t)$ for $T_{opt} - \tau_{k+m} - \tau_{k} \le t \le  T_{opt} - \tau_{k+m}$,  \dots, $u_{m+1}(1, t)$ for $T_{opt} - \tau_{m+1} - \tau_{1} \le t \le  T_{opt} - \tau_{m+1}$
 such that $\big(u_{m+1}(\cdot, 0), \cdots, u_{m+k}(\cdot, 0) \big)\tr = U$ (this can be done by the form of $S_{++}$)} and  chooses $u_{l}(t, 1)$ for $T_{opt} - \tau_{l} \le t \le T_{opt}$ and $m+1 \le l \le k + m$ in such a way that
\begin{equation}
\label{final-arrangement}
u_l(T_{opt}, x) = 0  \mbox{ for } x \in (0, 1).
\end{equation}
\newtextb{Requirement \eqref{final-arrangement} is again} possible by the property  of $S_{++}$ in \eqref{cond-S} and
by the form of $S$ in \eqref{def-S}.

\newtextb{\begin{remark} \label{rem-necessary-cond}  The above analysis shows that the existence of a bounded solution $U$ of \eqref{equation-U} implies the existence of a control to steer the system from the initial  data to $0$ in time $T_{opt}$ by the characteristic method. Moreover, in the case where $m = k$ and 
$$
T_{opt} = \tau_1 + \tau_{m+1} =  \dots =  \tau_k + \tau_{m+k},
$$
the existence of such a $U$ is necessary.
\end{remark}}

\newtext{\begin{remark} \label{rem-proof-thm-exact}We now show how to modify the proof of Theorem~\ref{thm1} in the case $m \ge k \ge 1$ to reach the exact controllability.
To obtain the exact controllability with the final state $v$, the requirements in $a_1, \dots, a_k)$ become \\
\noindent $c_1)$
$$
u_{k}(T_{opt}, x) = v_k(x) \mbox{ for }  x \in (x_{0, k}, x_{1, k}), \quad
\cdots, \quad  u_{1}(T_{opt}, x) = v_1(x) \mbox{ for }  x \in (x_{0, 1}, x_{1, 1}).
$$
\medskip
\noindent $c_2)$
$$
u_{k}(T_{opt}, x) = v_{k}(x) \mbox{ for } x \in  (x_{1, k}, x_{2, k}), \quad
\cdots, \quad  u_{2}(T_{opt}, x) = v_2(x) \mbox{ for }  x \in (x_{1, 2}, x_{2, 2}).
$$
\dots\\
\medskip
\noindent $c_k)$
$$
u_{k}(T_{opt}, x) = v_k(x) \mbox{ for } x \in  (x_{k-1, k}, x_{k, k}).
$$
Equations \eqref{cond-a1}, \eqref{cond-a2}, and \eqref{cond-ak} then become
\begin{equation*}
U_1(t) = J_1(t) + A_1  V_m(t) + \int_t^{t_0} G_{1}(t, s) V_m (s) \, ds  + \int_t^{t_0} H_{1}(t, s) U_1(s) \, ds \quad \mbox{ for } t_1 \le t \le t_0,
\end{equation*}
\begin{equation*}
U_{2}(t) =   J_2(t) + A_{2} V_{m+1}(t) +  \int_t^{t_0} G_{2}(t, s) V_{m+1}(s) \, ds  + \int_t^{t_0} H_{2}(t, s) U_{2}(s) \, ds \quad \mbox{ for } t_2 \le t < t_1,
\end{equation*}
\begin{equation*}
U_k(t) =  J_k(t) + A_{k} V_{m+k-1}(t)  + \int_t^{t_0} G_{k}(t, s) V_{m+k - 1}(s) \, ds  + \int_t^{t_0} H_{k}(t, s) U_k(s) \, ds \quad \mbox{ for } t_k \le t < t_{k-1},
\end{equation*}
for some functions $J_1, J_2, \dots, J_k$ depending on the final state $v$. Using  (\ref{cond-ck}-\ref{cond-c1}), one can write these equations  under the form
\begin{equation*}
U + {\cal K} (U) = F \mbox{ in } {\cal X}.
\end{equation*}
where $F$ now also depends on $J_1, \dots, J_k$. The rest of the proof of the exact controllability is unchanged.
\end{remark}}

\newtextb{\begin{remark}\label{rem-OPT} \rm In the case  where $m = k$ and 
$$
T_{opt} = \tau_1 + \tau_{m+1} =  \dots =  \tau_k + \tau_{m+k},
$$
the above analysis also gives the optimality of $T_{opt}$ for all $\gamma$ such that $I + {\cal K}$ is invertible. Indeed,  assume that there exists $T < T_{opt}$ such that one can steer an arbitrary state $u(0, \cdot)$ to 0 at the time $T$. Without loss of generality, one might assume that $T_{opt} -T$ is small.  To simplify the notations, we assume that $\Sigma$ is constant. As mentioned in Remark~\ref{rem-necessary-cond}, a necessary condition to have a control is the existence of a solution $U \in {\cal X}$ of 
\begin{equation}\label{tototo}
U  + {\cal K} U = G, 
\end{equation}
where $G$ now depends on $u_{i}(0, x)$ for $1 \le i \le k$ and $x \in (0, 1)$ and $u_i(0, x)$ for $k+1 \le i \le k+m$ and $x \in (0, 1-s_i)$ with $s_i = (T_{opt} - T)/\lambda_i$  by (\ref{cond-ck}-\ref{cond-c1}). However, for $t \in (1/ \lambda_{k+m} - (T_{opt} - T), 1/ \lambda_{k+m}) $,  
\begin{equation}\label{TTT}
u_{m+k}(t, 0) = u_{k+m}(0, \lambda_{k+m} t),
\end{equation}
the LHS  of \eqref{TTT} is uniquely determined by $G$ from \eqref{tototo} and  the RHS  of \eqref{TTT} can be chosen independently of $G$.  This yields a contradiction. 
\end{remark}}

\subsubsection{On the case $m < k$}\label{sect-m<k}

Set
$$
\hat u(t, x) = \big(u_{k-m +1},  \dots u_{k + m} \big)\tr (t, x) \mbox{ in } (0, T) \times (0, 1),
$$
$$
\hat \Sigma(x) = \mbox{ diag } (-\lambda_{k-m+1}, \cdots, - \lambda_k, \lambda_{k+1, } \cdots, \lambda_{m+k})(x)  \mbox{ in }  (0, 1),
$$
and denote
$$
\mbox{$\hat S(x)$  the $2m \times 2m $ matrix formed from the last $2 m$ columns and the last $2 m$ rows of $S(x)$},
$$
and
$$
\mbox{$\hat B$ the $m \times m$ matrix formed from the last $m$ rows of $B$}.
$$
Then $\hat u$ is a bounded  broad solution of the system
\begin{equation}\label{sys-hatS}
\partial_t \hat u(t, x) = \hat \Sigma(x) \partial_x \hat u(t, x) + \hat S(x) \hat u(t, 0),
\end{equation}
with the boundary condition at $0$ given by $(\hat u_1, \cdots, \hat u_{m})(t, 0)\tr  = \hat B (\hat u_{m+1}, \cdots, \hat u_{2m})(t, 0)\tr $.
Set
$$
\hat T_{opt}: = \max\{\tau_{k + m} + \tau_k, \cdots, \tau_{k+1} + \tau_{k+1 -m}\} = T_{opt}.
$$
Consider the pair $(\gamma, \hat B)$ such that the control  constructed in Section~\ref{sect-m>=k}  for $\hat u$ exists.
Then, for this control,
\begin{equation}\label{m<k-1}
\hat u (T_{opt}, x) = 0 \mbox{ for } x \in (0, 1).
\end{equation}
As observed in Remark~\ref{rem-m=k}, one has
$$
(\hat u_{m+1}, \dots, \hat u_{2m})\tr  (t, 0) = 0 \mbox{ for } t \in [T_{opt} - \tau_{k-m + 1}, T_{opt}].
$$
This yields
\begin{equation}\label{m<k-2}
(u_1, \dots, u_{k-m})\tr  (T_{opt}, x) = 0 \mbox{ for } x \in (0, 1)
\end{equation}
by the form of $S$ given in \eqref{def-S}.

Combining \eqref{m<k-1} and \eqref{m<k-2} yields the null-controllability at the time $T_{opt}$.  \qed

\subsection{Proof of part 2) of Theorem~\ref{thm1}}

\newtextb{Fix $\alpha\neq 0$ and $\beta \neq 0$, and  consider 
\begin{equation}\label{form-K-C-3}
C^{3,3} =   \left(\begin{array}{cccc}
0 & 0 & \alpha (\lambda_{k+2} + \lambda_k) \\[6pt]
0 & 0 & \beta (\lambda_{k+2} - \lambda_{k+1}) \\[6pt]
0 & 0 & 0
\end{array} \right) \mbox{ and } 
C(x)=   C^{k+2, k+2} := \left(\begin{array}{cccc}
0_{k-1, k-1} & 0_{k-1, 3}\\[6pt]
0_{3, k - 1} & C^{3,3}
\end{array}\right) \mbox{ for } k \ge 1.
\end{equation}
Set
\begin{equation}\label{form-K-C}
K^{3, 3} = \left(\begin{array}{cccc}
 0 & 0 & \alpha \\[6pt]
 0 & 0 & \beta \\[6pt]
 0 & 0 & 0
\end{array}\right)
 \quad \mbox{ and } \quad
K(x)=   K^{k+2, k+2} := \left(\begin{array}{cccc}
0_{k-1, k-1} & 0_{k-1, 3}\\[6pt]
0_{3, k - 1} & K^{3,3}
\end{array}\right) \mbox{ for } k \ge 1.
\end{equation}}

 One can check that $K C = 0_{k+2}$ and $K$ is a solution of equation~\eqref{equation-K} by noting that $\Sigma$ is constant. Moreover, \eqref{bdry1-K}  holds by  the choice of $K$ and $C$. We have, by \eqref{def-S},   that
\begin{equation}\label{form-S-K-C}
S(x)=   S^{k+2, k+2} = \left(\begin{array}{cccc}
0_{k-1, k-1} & 0_{k-1, 3}\\[6pt]
0_{3, k - 1} & S^{3,3}
\end{array}\right)
 \mbox{ where } S^{3,3}= \left(\begin{array}{cccc}
 0 & 0 & \lambda_{k+2}  \alpha \\[6pt]
 0 & 0 & \lambda_{k+2} \beta \\[6pt]
 0 & 0 & 0
\end{array}\right).
\end{equation}
In what follows, for simplicity of notations, we only consider the case $k =2$. The other cases can be established similarly.
Suppose  that
\begin{equation}\label{cond-u2}
u_2(t, 0) = a u_{3}(t, 0) + b u_{4}(t, 0) \mbox{ for } t \ge 0.
\end{equation}
\newtextb{Then $a \neq 0$ by  condition~\eqref{cond-B-1}.}
To obtain the null-controllability at the time $T_{opt}$, one  has, from condition $a_1)$,
$$
u_{3}(t, 0) = u_{4}(t, 0) = 0 \mbox{ for } t \in (t_1, t_0)
$$
and, hence from condition $a_2)$ and \eqref{form-S-K-C},
$$
a u_3(t, 0) + b u_4(t, 0) + \int_t^{t_1} \lambda_4 \alpha u_4(s, 0) \, ds = 0 \mbox{ for } t \in (t_2, t_1).
$$
Since, by \eqref{eq-u},  \eqref{def-S-*}, and \eqref{form-S-K-C},
$$
u_3(t, 0) = \int_{t_2}^t \lambda_4 \beta u_4(s, 0) \, ds + f (t) \mbox{ for } t \in (t_2, t_1),
$$
for some $f$ depending on the initial data. By taking $ \beta = \alpha/a$ \newtextb{($a \neq 0$)},  we have
\begin{equation}\label{t1t2-f}
b u_4 (t, 0) + \int_{t_2}^{t_1} \lambda_4 \alpha  u_4(s, 0) \, ds = - a f (t) \mbox{ for } t \in (t_2, t_1).
\end{equation}
By choosing $\alpha$ such that $b + \lambda_4 \alpha (t_1 - t_2)  = 0$, \newtextb{and integrating \eqref{t1t2-f} from $t_2$ to $t_1$,  it follows, since $a \neq 0$, that}
\begin{equation}\label{xxx}
\int_{t_2}^{t_1} f(t) \, dt = 0.
\end{equation}
This is impossible for an arbitrary initial data, for example if $u_4(0, \cdot) = 0$ then $f(t) = u_3 (0, \lambda_3 t)$ and an appropriate choice of $u_3(0, \cdot)$ yields that \eqref{xxx} does not hold. In other words, the system is not null-controllable at the time $T_{opt}$.  \qed

\section{A null-controllability result for all $\gamma$ and $B \in {\cal B}$} \label{sect-FP}

 A slight modification of the proof of part 3) of Theorem~\ref{thm1} gives the following result, where $T_2$ is defined in \eqref{def-T2}.
\begin{proposition}\label{pro-T2} Let $m \ge 2$. Assume that \eqref{relation-lambda} and \eqref{cond-lambda} hold and $B \in {\cal B}$. There exists $\delta > 0$ depending only on $C, B, \Sigma$, and $\gamma$ such that the system  is null-controllable at the time $T_{2} - \delta$.
\end{proposition}

\begin{proof} We only consider here the case $m \ge k$, the case $m < k$ can be handled  as in Section~\ref{sect-m<k}. The controls are chosen so that
\begin{equation}\label{FP-1}
u_{k+1}(t, 0)  = 0 \mbox{ for } t \ge \tau_{k+1}, \quad \dots,  \quad u_{k+m-1} (t, 0) = 0 \mbox{ for } t \ge \tau_{k+ m -1},
\end{equation}
\begin{equation}\label{FP-2}
u_{k+m}(t, 0)  = 0 \mbox{ for } (t \ge \tau_{k+m} \mbox{ and } t \not \in [\tau_{k+1} - \delta, \tau_{k+1}]),
\end{equation}
and $u_{k+m}(t, 0)$ is chosen in $[\tau_{k+1} - \delta, \tau_{k+1}]$ in such a way that
\begin{equation}\label{FP-3}
u_{k}(T_2 - \delta, x) = 0 \mbox{ for } x \in [x^*, 1],
\end{equation}
where $x^* =  x_{k}(T_2 - \delta, \tau_{k+1}, 0)$ (see the definition of $x_k$ in \eqref{def-xi-1}). As in $b_k)$,  we derive that,  in $[\tau_{k+1} - \delta, \tau_{k+1}]$,  \eqref{FP-3} is equivalent to
\begin{equation}\label{fixed-point}
u_{k+m}(t, 0) = \int_{\tau_{k+1}  - \delta}^{\tau_{k+1}} K(t, s) u_{k+m}(s, 0) \, ds + f(t),
\end{equation}
for some bounded function $f$ defined in $[\tau_{k+1} - \delta, \tau_{k+1}]$ which now depends only on the initial data. Here $K: [\tau_{k+1} - \delta, \tau_{k+1}]^2 \to \mR$ is a bounded function depending only on $C, B, \Sigma$, and $\gamma$. Since $\delta$ is small, one can check that the mapping $T: L^2([\tau_{k+1} - \delta, \tau_{k+1}]) \to L^2([\tau_{k+1} - \delta, \tau_{k+1}])$ which is  given by
$$
T(v)(t) = \int_{\tau_{k+1}  - \delta}^{\tau_{k+1}} K(t, s) v(s) \, ds
$$
is a contraction. By the contraction mapping theorem, equation \eqref{fixed-point} is uniquely solvable and the solution is bounded since $f$ is bounded.

We now show how to construct such a control. Since $u_{m+k}(t, 0)$ for $0 \le t \le \tau_{m+k}$ is  uniquely determined by the initial data (by \eqref{cond-S}), one derives from  \eqref{cond-S} that $u_{m+k-1}(t, 0)$ for $0 \le t \le \tau_{m+k -1}$, \dots, $u_{k+1}$ for $0 \le t \le \tau_{k+1} $ are uniquely determined from the initial condition and the requirements on the constructive controls at $(t, 0)$.
It follows from    \eqref{cond-S} again  that
\begin{itemize}
\item $u_{k+m}(t, 1)$ for $t \ge 0$ is uniquely determined from $u_{m+k}(t, 0)$ for $t \ge \tau_{m+k}$,
\item  $u_{m+k-1}(t, 1)$ for $t \ge 0$ is uniquely determined from ($u_{m+k-1}(t, 0)$ for $t \ge \tau_{m+k -1}$ and  $u_{m+k}(t, 0)$ for $t \ge 0$),

\dots,
\item $u_{k+1}(t, 1)$ for $t \ge 0$ is uniquely determined from
($u_{k+1}(t, 0)$ for $t \ge \tau_{k+1}$, $u_{k+2}(t, 0)$ for $t \ge 0$, \dots,   $u_{m+k}(t, 0)$ for $t \ge 0$).
\end{itemize}
The existence and uniqueness of  controls satisfying requirements are established.

It remains to check that the constructive controls give the null-controllability at the time $T_2 - \delta$ if $\delta$ is small enough.  Indeed, by \eqref{FP-1} and \eqref{FP-2}, we have
\begin{equation}\label{FP-4}
u_{k+1}(t, 0) = \dots = u_{k + m}(t, 0) = 0 \mbox{ for } t \ge \tau_{k+1}.
\end{equation}
Since   $S_{--} = 0_{k}$, it follows from  \eqref{bdry:x=0-w} that
\begin{equation}\label{FP-5}
u_{1}(T_{2} - \delta, x) = \dots = u_{k-1}(T_{2} - \delta, x) = 0 \mbox{ for } x \in [0, 1]
\end{equation}
and
\begin{equation*}
u_{k}(T_{2} - \delta, x) =   0 \mbox{ for } x \in [0, x^*],
\end{equation*}
which yields, by \eqref{FP-3},
\begin{equation}\label{FP-6}
u_{k}(T_{2} - \delta, x) =   0 \mbox{ for } x \in [0, 1].
\end{equation}
From \eqref{FP-4} and the form of $S$, we also derive that
\begin{equation*}
u_{k+m}(t, x) = \dots = u_{k+1}(t, x) \mbox{ for } x \in [0, 1] \mbox{ and } t \ge \tau_{k+1};
\end{equation*}
in particular, if $\delta$ is small enough,
\begin{equation}\label{FP-7}
u_{k+m}(T_2 - \delta, x) = \dots = u_{k+1}(T_2 - \delta, x) \mbox{ for } x \in [0, 1].
\end{equation}
The null-controllability at $T_2 - \delta$ now follows from \eqref{FP-5}, \eqref{FP-6}, and \eqref{FP-7}.
\end{proof}

\section{On  the case $m=2$ and $B_{k1} \neq 0$ - Proof of Theorem~\ref{thm2}} \label{sect-thm2}

\newtext{We only establish the null-controllability result. The proof of the exact controllability can be derived similarly as in the spirit mentioned in Remark~\ref{rem-proof-thm-exact} and is omitted.}
Without loss of generality, one might assume that $\gamma = 1$ and $T - T_{opt}$ is small. As mentioned in Remark~\ref{rem-choice}, the choice of $K$ on $\Gamma_3$ in \eqref{bdry-K-3} can be ``arbitrary".  In this section, we modify this choice to reach some analytic property of  $K$. The new $K$ will be defined in $\hat \D$ which is the triangle formed by three points $(0, 0)$, $(1, 0)$,  and $(L, L)$, where $L$ is defined in \eqref{def-L},  this triangle contains $\D$. 
Since $B_{k1} \neq  0$, by \eqref{def-S}, one can replace the  condition $K_{k k} = 0$ in \eqref{cond-S-0} by the condition
$$
\big(S_{-+}\big)_{k1} = 0
$$
while the rest of \eqref{cond-S-0} remains unchanged.
The idea of the proof is to show that one can prepare $u(T- T_{opt}, \cdot)$ using the control in the time interval $[0, T- T_{opt}]$ in such a way that \eqref{equation-U} is solvable. In what follows, we present a direct proof for Theorem~\ref{thm2}.

We first consider the case $k = m $ (=2).  The matrix $S$ then has the form
\begin{equation}\label{form-S22}
S = \left(\begin{array}{cccc}
0 & 0 & * & * \\[6pt]
0 & 0 & 0 & * \\[6pt]
0 & 0 & 0 & * \\[6pt]
0 & 0 & 0 & 0
\end{array}\right).
\end{equation}
We choose a control so  that $u_{3}(t, 0) = u_{4}(t, 0) = 0$ for $ t  \ge T - \tau_1$, $u_4(t, 0) = 0$ for $\tau_4 \le t \le T - \tau_2 $, and $u_3(t, 0) = 0$ for $\tau_3 \le t \le T_{opt}- \tau_1$ (the last two choices are just a preparation step), and  as in $a_2)$,  $u_4(t, 0)$ for  $t \in (T - \tau_2, T - \tau_1)$ is required to ensure that
\begin{equation}\label{coucou}
u_2(T, x) = 0 \mbox{ for } x \in [x_{12}, 1]
\end{equation}
(see \eqref{def-xij} for the definition of $x_{12}$).
One can verify that the null-controllability is attained at $T$ for such a control if it exists. As in the proof of Proposition~\ref{pro-T2}, it suffices to show  that \eqref{coucou} is solvable for some (bounded) choice of $u_3(t, 1)$ for $t \in (T_{opt} - \tau_1 - \tau_3, T - \tau_1 - \tau_3)$.  Let $\chi_O$ denote a characteristic function of a subset $O$ of $\mR$.
By the form of $S$ in \eqref{form-S22} and the fact that $B \in {\cal B}$, equation \eqref{coucou} is equivalent to, for  $t \in (T - \tau_2, T - \tau_1)$,
 \begin{equation}\label{eq-K1}
h(t) + u_4(t, 0) + \alpha u_3(t, 0) \chi_{[T_{opt} - \tau_1, T-\tau_1]}= \int_{T -\tau_2}^t g(t-s) u_4(s, 0) \, ds + \int_t^{T - \tau_1} f(s-t) u_4(s, 0) \, ds,
 \end{equation}
 where $g$ and $f$ are two functions depending only on $K$, $B$, and $\Sigma$, $\alpha$ is a non-zero constant (since $B_{21} \neq 0$), and $h(t)$ is a function now depends only on $B$, $K$, $\Sigma$, and the initial condition. Moreover, $f$ and $g$ are analytic by Lemma~\ref{lem-A2} below.   Let $
 {\cal K}_1: L^2(T - \tau_2, T - \tau_1) \to L^2(T - \tau_2, T - \tau_1)
 $ be defined by the RHS of \eqref{eq-K1}. Then the adjoint operator  ${\cal K}_1^*:  L^2(T - \tau_2, T - \tau_1)  \to L^2(T - \tau_2, T - \tau_1)$ is given by
 \begin{equation*}
{\cal K}_1^*(v) = \int_{t}^{T - \tau_1} g(s-t) v(s) \, ds  + \int_{T - \tau_2}^{t} f(t-s) v(s) \, ds.
 \end{equation*}

Let $V$ be an eigenfunction of ${\cal K}_1^*$ with respect to the eigenvalue $-1$. We have, by Lemma~\ref{lem-A1} below,
\begin{equation}\label{claim1-thm2}
V \not \equiv 0 \mbox{ in a neighborhood of $T - \tau_1$}.
\end{equation}
Since  the kernel of $I + {\cal K}_1^*$ is of finite dimension, one can prepare the state at the time $T- T_{opt}$ (i.e. $u_3(t, 1)$ for $t \in (T_{opt} - \tau_1 - \tau_3, T- \tau_1 - \tau_3)$) in such a way that the RHS of \eqref{eq-K1} is orthogonal to the kernel of $I + {\cal K}_1^*$. It follows from the Fredholm theory  that \eqref{eq-K1} is solvable and the solution is bounded.

We next consider the case $k > m=2$. The proof in this case follows from the previous one as in Section~\ref{sect-m<k}. We finally consider the case $k=1$ and $m =2$.  The matrix $S$ then has the form
\begin{equation}\label{form-S12}
S = \left(\begin{array}{cccc}
 0 & 0 & * \\[6pt]
 0 & 0 & * \\[6pt]
 0 & 0 & 0
\end{array}\right).
\end{equation}
We choose  a control such that  $u_3(t, 0) = 0$ for $\tau_3 \le t \le T - \tau_1$,  $u_2(t, 0) =  0$ for $\tau_2 \le t \le T_{opt}$ (a preparation step),  $u_2(t, 0) = u_3(t, 0) =0 $ for $t \ge T$,  and  $u_3(t, 0)$ for  $t \in (T- \tau_1, T)$ is required to ensure that
\begin{equation}\label{coucou2}
u_1(T, x) = 0 \mbox{ for } x \in [0, 1].
\end{equation}
One can verify that the null-controllability is attained at $T$ for such a control if it exists. As in the proof of Proposition~\ref{pro-T2}, it suffices to show that \eqref{coucou2} is solvable.
As before,   equation \eqref{coucou2} is equivalent to, for $t \in (T- \tau_1, T)$,
 \begin{equation*}
h(t) + u_3(t, 0) + \alpha u_2(t, 0) \chi_{[T_{opt} , T]} = \int_{T - \tau_1}^t g(t-s) u_3(s, 0) \, ds + \int_t^{T} f(s-t) u_3(s, 0) \, ds,
 \end{equation*}
 where $g$ and $f$ are two functions depending only on $K$, $B$, and $\Sigma$, $\alpha$ is a non-zero constant (since $B_{21} \neq 0$), and $h(t)$ is a function now depends only on $B$, $K$, $\Sigma$, and the initial condition.   Moreover, $f$ and $g$ are analytic by Lemma~\ref{lem-A2} below.  The proof now follows as in the case $k = m = 2$ and the details are omitted.
\proofend

\medskip
The following result is used in the proof of Theorem~\ref{thm2}.

%
%
%
%

\begin{lemma} \label{lem-A1} Let  $T>0$, $f, g \in C^1([0, T])$ and let $V$ be a continuous function defined in $[0, T]$  such that
\begin{equation*}
V(t)  = \int_0^t f(t - s) V(s)  + \int_t^T g(s -t) V(s) \, ds \mbox{ for } t \in [0, T].
\end{equation*}
Assume that $g$ is analytic on $[0, T]$ and $V = 0 $ in a neighbourhood of $0$. Then $V \equiv 0$ on $[0, T]$.
\end{lemma}

\begin{proof} It suffices to prove that $V$ is analytic on $[0, T]$. We have
\begin{equation}\label{lem-A1-1}
V'(t) = f(0) V(t) + \int_0^t f'(t-s) V(s) \, ds - g(0) V(t) - \int_t^T g'(s-t) V(s) \, ds.
\end{equation}
Since $V = 0 $ in a neighborhood of $0$, an integration by parts gives
\begin{equation}\label{lem-A1-2}
V'(t) = \int_0^t f(t - s) V'(s)  + \int_t^T g(t -s) V'(s) \, ds - g(T -t) V(T).
\end{equation}
By recurrence, we obtain, for $n \ge 0$,
\begin{align}\label{lem-A1-3}
V^{(n+1)}(t) =  & f(0) V^{(n)}(t) + \int_0^t f'(t-s) V^{(n)}(s) \, ds - g(0) V^{(n)}(t) - \int_t^T g'(s-t) V^{(n)}(s) \, ds \nonumber \\[6pt]
& + \sum_{k=0}^{n-1} (-1)^{n-k+1} g^{(n-k)}(T-t) V^{(k)}(T)
\end{align}
and
\begin{equation}\label{lem-A1-4}
V^{(n+1)}(t) = \int_0^t f(t - s) V^{(n+1)}(s)  + \int_t^T g(t -s) V^{(n+1)}(s) \, ds  + \sum_{k=0}^{n} (-1)^{n-k+1} g^{(n-k)}(T-t) V^{(k)}(T).
\end{equation}
By rescaling, without loss of generality, one might assume that
$$
T = 1 \quad \mbox{ and } \quad \| V\|_{C^1([0, T])} = 1.
$$
Set
\begin{equation*}
a_n = \| V^{(n)} \|_{L^\infty([0, T])}  \quad \mbox{ and } \quad b_n  =  \| g^{(n)}\|_{L^\infty([0, T])} + \| g^{(n+1)} \|_{L^\infty([0, T])} + 2 \|f \|_{C^1([0, T])}.
\end{equation*}
Using \eqref{lem-A1-3}, we obtain
\begin{equation}\label{lem-A1-t1}
a_{n+1} \le \sum_{k=0}^n a_{n-k} b_{k}.
\end{equation}
We have, by the analyticity of $g$,
\begin{equation}\label{lem-A1-t2}
b_k \le c^k k!.
\end{equation}
In this proof, $c$ denotes a constant greater than 1 and  independent of $k$ and $n$. It is clear that
\begin{equation}\label{lem-A1-t3}
\sum_{k=0}^n c^k k! c^{n-k} (n-k)! \le c^{n} (n+1)!,
\end{equation}
Combining \eqref{lem-A1-t1}, \eqref{lem-A1-t2}, and \eqref{lem-A1-t3} and using a recurrence argument yield
$$
a_n \le c^n n!.
$$
The analyticity of $V$ now follows from the definition of $a_n$. The proof is complete.
\end{proof}

The second lemma yields the analyticity of $g$ and  $f$ in the definition of ${\cal K}_1^*$ in the proof of Theorem~\ref{thm2}.

\begin{lemma} \label{lem-A2} Let $l \ge  4$ and $\gamma_1 \le \gamma_2 \le \cdots  \le \gamma_l$ be such that ${\cal I}_1, {\cal I}_2, {\cal I}_3, {\cal I}_4 \neq \emptyset$,  where
$$
{\cal I}_1 = \{ i : \; \gamma_i \le 0\}, \quad {\cal I}_2 = \{i: \; 0 < \gamma_i <  1\},  \quad {\cal I}_3 = \{i: \;\gamma_i = 1\}, \quad {\cal I}_4 = \{i: \; \gamma_i > 1\}.
$$
Denote $\hD$ the triangle formed by three lines $y = x$, $y = 0$, and $y = \gamma_{i_0} x - \gamma_{i_0}$ where $i_0 = \min {\cal I}_4$. 
Let $G : [0, \gamma_{i_0}/ (\gamma_{i_0} - 1)] \to \mR^{n \times n}$ be analytic,  and denote $\Gamma = \mbox{diag} (\gamma_1, \cdots, \gamma_l)$
and
$$
\Lambda: = \{ (x, y) \in \partial \hD; x = y \} = \{ (x, x); x \in [0, \gamma_{i_0}/ (\gamma_{i_0} - 1)] \}.
$$
Let  $f_i$ $(i \in {\cal I}_1 \cup {\cal I}_4)$ be analytic functions defined in a neighborhood of  $\Lambda$
and let $c_{i, j} \in \mR$ for $1 \le i, j \le l$. Assume that  $v$ is  the unique broad solution of the system
\begin{equation*}
\left\{\begin{array}{c}
\partial_x v (x, y ) + \Gamma \partial_y v (x, y)  - G(y) v (x, y) = 0 \mbox{ in } \hD, \\[6pt]
v_{i}(x, x) = f_i(x, x)  \mbox{ for } (x, x) \in \Lambda, \, i \in {\cal I}_1  \cup {\cal I}_4,  \\[6pt]
v_{i} (x, 0 )= \sum_{j \in {\cal I}_1  \cup {\cal I}_4 } c_{ij} v_{j} (x, 0) \mbox{ for } x \in (0, 1), \, i \in {\cal I}_2  \cup {\cal I}_3.
\end{array}\right.
\end{equation*}
Then $v$ is analytic in $\bar \Delta_{i+1} \setminus \Delta_i$ for $i \in \hat {\cal I}_2$ where $\hat I_2 = \big\{i \in {\cal I}_2 \mbox{ or } i + 1 \in {\cal I}_2 \big\}$ and $\Delta_j$ is the open triangle formed by three lines  $y = \gamma_j x$, $y = 0$, and $y = \gamma_{i_0} x - \gamma_{i_0}$.
\end{lemma}

%
%
%
%
%
%
%
%
%
%
%
%
%
%
%

\begin{proof} We first prove by recurrence that, for $k \ge 1$,
\begin{equation}\label{recurrence}
\|v\|_{C^k(\bar \Delta_{i+1} \setminus \Delta_i)} \le C^k \|f\|_{C^k(\Gamma)} \mbox{ for } i \in \hat {\cal I}_2.
\end{equation}
In this proof, $C$ denotes a positive constant independent of $k$ and $f$.
Indeed, using the standard fixed point iteration, one can show that $v \in C^{1} (\bar \Delta_{i+1} \setminus \Delta_i)$ ($i \in \hat {\cal I}_2$);  moreover,
\begin{equation}
\| v\|_{C^1(\bar \Delta_{i+1} \setminus \Delta_i)} \le  C \| f\|_{C^1(\Gamma)} \mbox{ for } i \in \hat {\cal I}_2.
\end{equation}
Hence \eqref{recurrence} holds for $k=1$.  Assume that \eqref{recurrence} is valid for some $k \ge 1$. We prove that  it holds  for $k+1$. Set
$$
V = \partial_{x} v \mbox{ in } \hD.
$$
We have
\begin{equation*}
\left\{\begin{array}{c}
\partial_x V (x, y)+ \Gamma \partial_y V(x, y)  - G(y) V (x, y) = 0  \mbox{ in } \hD, \\[6pt]
V_{i}(x, x) = g_i(x, x)  \mbox{ for } (x, y) \in \Lambda,  i \in {\cal I}_1  \cup {\cal I}_4,  \\[6pt]
V_{i} = \sum_{j \in {\cal I}_1  \cup {\cal I}_4 } c_{ij} V_{j}  \mbox{ for } x \in (0, 1), \mbox{ for } i \in {\cal I}_2  \cup {\cal I}_3,
\end{array}\right.
\end{equation*}
where
$$
g_i(x,x) = a_i G(x) v_i(x, x) + b_i \frac{d}{dx} [f_i (x, x)]
$$
for some positive constant $a_i, b_i \in \mR$ depending only on $\Gamma$.  This is obtained by considering the first equation and the derivative with respect to $x$ of the second equation in the system of $v$.  By the recurrence, one has
$$
\| u \|_{C^k(\bar \Delta_{i+1} \setminus \Delta_i)}   \le  C^k \| g \|_{C^k(\Gamma)} \mbox{ for } i \in \hat {\cal I}_2.
$$
Using the equation of $v$, one derives that
$$
\|u\|_{C^{k+1}(\bar \Delta_{i+1} \setminus \Delta_i)} \le C^{k+1} \|f\|_{C^{k+1}(\Gamma)} \mbox{ for } i \in \hat {\cal I}_2.
$$
Assertion \eqref{recurrence} is established.

The conclusion now follows from the analyticity of $f$.
\end{proof}

\section{On the case $C \equiv 0$ - Proof of Proposition \ref{pro-C}} \label{sect-pro-C}

Note that $S \equiv 0$ since  $C \equiv 0$.  We first construct a time independent feedback to reach the null-controllability at the time $T_{opt}$.  We begin with  considering the case $m >  k$.  Condition  $a_k)$  can be written under the form
\begin{equation}\label{kd-1}
u_{m+k}(t, 0) = M_k (u_{k +1 }, \cdots, u_{m + k - 1})\tr (t, 0)  \quad \mbox{ for } t \in (t_k, t_{k-1}),
\end{equation}
for some  constant  matrix  $M_k$ of size $1 \times (m  -1)$ by considering \eqref{cond-B-1} with $i= 1$.  Condition $a_{k-1})$ can be written under the form \eqref{kd-1} and
\begin{equation}\label{kd-2}
u_{m+k-1}(t, 0) = M_{k-1} (u_{k +1 }, \cdots, u_{m + k - 2})\tr (t, 0)  \quad \mbox{ for } t \in (t_k, t_{k-1}),
\end{equation}
for some  constant  matrix  $M_{k-1}$ of size $1 \times (m  -2)$ by applying \eqref{cond-B-1} with $i= 2$ and using the Gaussian elimination method, etc. Finally, condition $a_1)$ can be written under the form \eqref{kd-1}, \eqref{kd-2}, \dots, and
\begin{equation}\label{kd-3}
u_{m+1}(t, 0) = M_{1} (u_{k +1 }, \cdots, u_{m})\tr (t, 0)  \quad \mbox{ for } t \in (t_k, t_{k-1}),
\end{equation}
for some  constant  matrix  $M_{2}$ of size $1 \times (m  - k)$ by applying \eqref{cond-B-1} with $i= k$ and using the Gaussian elimination method if $m>k$; this condition is replaced by the one $u_{m+1} = 0$ in the case $m = k$. \newtextb{The matrices $M_{1}, \cdots, M_{k}$ can be obtained via the Gaussian elimination method starting  with $M_1$ using condition \eqref{cond-B-1} with $i=1$, and then with $M_2$ using  condition \eqref{cond-B-1} with $i=2$, $\cdots$,  and finally with $M_k$ using  condition \eqref{cond-B-1} with $i=k$.}

We now choose the following feedback law
\begin{equation}\label{bdry-1}
u_{m+ k}(t, 1) = M_{k} \Big(u_{k+ 1}\big(t, x_{k+1} (- \tau_{m + k}, 0,  0) \big), \dots, u_{k+ m-1}\big(t, x_{k+ m - 1} (-\tau_{m + k }, 0,  0) \big)\Big),
\end{equation}
\begin{equation}\label{bdry-2}
u_{m+ k - 1}(t, 1) = M_{k-1} \Big(u_{k+ 1}\big(t, x_{k+1} (- \tau_{m + k - 1}, 0,  0) \big), \dots, u_{k+ m-2}\big(t, x_{k+ m - 2} (-\tau_{m + k -1}, 0,  0) \big) \Big),
\end{equation}
\dots
\begin{equation}\label{bdry-3}
u_{m+ 1}(t, 1) = M_{1} \Big(u_{k+ 1} \big(t, x_{k+1} (- \tau_{m + 1}, 0,  0) \big), \dots, u_{m}\big(t, x_{m +1} (-\tau_{m +1}, 0,  0) \big)\Big)
\end{equation}
(this condition is replaced by the one $u_{m+1}(t, 1) = 0$ in the case $k = m$),
and
\begin{equation}\label{bdry-4}
u_{k+1}(t, 1) = \dots = u_{m}(t, 1) = 0.
\end{equation}
Let us point out that, by Lemma~\ref{lem-WP}, the closed-loop system of $u$ given by  $\partial_t u = \Sigma \partial_x u$ and the boundary conditions (\ref{bdry-1}-\ref{bdry-4})  is well-posed in the sense of Definition~\ref{def-broad}.
With this law of feedback, conditions $a_k)$, \dots, $a_1)$ hold.  It follows that
\begin{equation}\label{conclusion-kd-1}
u_{1}(T_{opt}, x) = \dots = u_{k}(T_{opt}, x)  = 0 \mbox{ in } (0, 1).
\end{equation}
We also derive from \eqref{bdry-4} \newtextb{using the characteristic method and the fact $C = 0$} that
\begin{equation*}
u_{k+1}(t, 0) = \dots = u_{m}(t, 0) = 0 \mbox{ for } t \ge \tau_{k + 1}
\end{equation*}
and from  (\ref{bdry-1}-\ref{bdry-3})  (see also (\ref{kd-1}-\ref{kd-3})) that
\begin{equation*}
u_{k+1}(t, 0) = \dots = u_{k + m}(t, 0) = 0 \mbox{ for }  t \ge T_{opt}.
\end{equation*}
We then obtain
\begin{equation}\label{conclusion-kd-2}
u_{k+1}(T_{opt}, x) = \dots = u_{k+m}(T_{opt}, x) = 0 \mbox{ for } x \in (0, 1).
\end{equation}
The null-controllability attained at the optimal time $T_{opt}$ now follows from \eqref{conclusion-kd-1} and \eqref{conclusion-kd-2}.

We next deal with the case $m < k$. The construction of a time independent feedback yielding a null-state at the time $t = T_{opt}$ in this case is based on the construction given in the case $m = k$ obtained previously. Set
$$
\hat u(t, x) = \big(u_{k-m +1},  \dots u_{k + m} \big)\tr (t, x) \mbox{ in } (0, T) \times (0, 1),
$$
$$
\hat \Sigma(x) = \mbox{ diag } (-\lambda_{k-m+1}, \cdots, - \lambda_k, \lambda_{k+1, } \cdots, \lambda_{m+k})(x) \mbox{ in }   (0, 1),
$$
and
$$
\mbox{$\hat B$ is the matrix formed from the last $m$ rows of $B$}.
$$
Then $\hat u$ is a bounded  broad solution of the system
\begin{equation*}
\partial_t \hat u(t, x) = \hat \Sigma(x) \partial_x \hat u(t, x),
\end{equation*}
with the boundary condition at $0$ given by $(\hat u_1, \cdots, \hat u_{m})(t, 0)\tr  = \hat B (\hat u_{m+1}, \cdots, \hat u_{2m})(t, 0)\tr $.
Consider the time dependent feedback for $\hat u$ constructed previously. Then, as in Section~\ref{sect-m<k}, the null-controllability is attained at $T_{opt}$ for this feedback. The details are omitted.

We next establish the second part of Proposition~\ref{pro-C} by contradiction. We  only deal with the case $m \ge k$. We first consider the case $T_{opt} = \max_{1 \le i \le k} \{ \tau_i + \tau_{i + m} \}$.  Fix $T \in \big(\max_{1 \le i \le k} \tau_{i + m}, T_{opt} \big)$  and let $1 \le i_0  \le k$ be such that $\tau_{i_0}  + \tau_{i_0 + m}= T_{opt}$.
Consider an initial datum $u$ such that $u_{i} (t = 0, x)=  0$ for  $x \in (0, 1)$ and for $1 \le i \neq i_0 + m \le k + m$ and $u_{i_0 + m}(t=0, x) = 1$ for $x \in (0, 1)$. Assume that  the null-controllability is attained at $T$.  By the convention of $\lambda_j$, one has, for some $\eps > 0$ depending on $\Sigma$,
$$
u_{i_0}(t, 0) = u_{i_0 + 1}(t, 0) = \dots = u_{k}(t, 0) = 0  \mbox{ for } t \in (T- \tau_{i_0}, T - \tau_{i_0} + \eps).
$$
As in \eqref{kd-1}, \eqref{kd-2}, and \eqref{kd-3}, we obtain, for $t \in (T- \tau_{i_0}, T - \tau_{i_0} + \eps)$,
\begin{equation*}
u_{m+k}(t, 0) = M_k (u_{k +1 }, \cdots, u_{m + k - 1})\tr (t, 0),
\end{equation*}
\begin{equation*}
u_{m+k-1}(t, 0) = M_{k-1} (u_{k +1 }, \cdots, u_{m + k - 2})\tr (t, 0),
\end{equation*}
\dots
\begin{equation}\label{kd-3-1}
u_{m+i_0}(t, 0) = M_{i_0} (u_{k +1 }, \cdots, u_{m  + i_0 - 1 })\tr (t, 0).
\end{equation}
Since $u_1(0, \cdot ) = \dots = u_{m+ i_0 - 1}(0, \cdot) = 0$, it follows form \eqref{kd-3-1} that
\begin{equation}\label{proC-1}
u_{m+i_0}(t, 0)  = 0 \mbox{ for } t \in (T- \tau_{i_0}, T - \tau_{i_0} + \eps).
\end{equation}
On the other hand, by using the characteristic method and the fact $T< \tau_{i_0} + \tau_{i_0 + m}$, one has, for $\eps $ small enough,
$$
u_{i_0 + m}(t, 0)  = 1 \mbox{ for } t \in (T- \tau_{i_0}, T - \tau_{i_0} + \eps).
$$
This contradicts \eqref{proC-1}. The second part of Proposition \ref{pro-C} is proved in this case.

We next consider the case $T_{opt} > \max_{1 \le i \le k} \{ \tau_i + \tau_{i + m} \}$. Then $T_{opt} = \tau_{k+1}$ and $m>k$.  The conclusion follows by considering $u_{i}(0, x) = 1$ for $1 \le i \neq k+1 \le k+m$ and $u_{k+1}(0, x ) = 1$.
\proofend

\medskip In what follows, we present two  concrete examples on the feedback form used in the context of Proposition~\ref{pro-C}. We first consider the case where $k = 1$, $m =2$,
$$
\Sigma_+ = \mbox{diag} (1, 2) \quad \mbox{ and } \quad B = (2, 1).
$$
One can check that \eqref{kd-1} has the form
\begin{equation*}
u_3(t, 0) = - 2 u_2(t, 0).
\end{equation*}
The feedback is then  given by
$$
u_3(t, 1) =  - 2 u_2(t, 1/2) \quad \mbox{ and } \quad u_2(t, 1) = 0 \mbox{ for } t \ge 0.
$$
We next consider the case where $k = 3$, $m = 3$,
$$
\Sigma_+ = \mbox{diag} (1, 2, 4) \mbox{ and the matrix formed from the last two rows of $B$ is } \left(\begin{array}{cccc}
2 & 0 & 1  \\[6pt]
-1 & -1 & 1
\end{array}\right).
$$
One can check that \eqref{kd-1} has the form (by imposing the condition $u_3(t, 0) = 0$)
\begin{equation*}
u_6(t, 0) = u_5(t, 0) + u_4(t, 0),
\end{equation*}
\eqref{kd-2} has the form (by imposing the condition $u_3(t, 0) = u_2(t, 0) = 0$)
\begin{equation*}
u_5(t, 0) = - 3 u_4(t, 0).
\end{equation*}
The feedback is then given by
$$
u_6(t, 1) =   u_5(t, 1/2) + u_4 (t, 1/4),  \quad  u_5(t, 1) = - 3 u_4(t, 1/2), \quad  \mbox{ and } \quad u_4(t, 1) = 0  \mbox{ for } t \ge 0.
$$
One can verify directly that  the null-controllability is reached for these feedbacks.


\end{document}